\documentclass[11pt,oneside]{amsart}
\usepackage{graphicx}
\usepackage[margin=1.25in]{geometry}
\usepackage[colorinlistoftodos]{todonotes}
\usepackage{amssymb}
\usepackage{amsmath}
\usepackage{amsthm}
\usepackage{amsfonts}
\usepackage{nicefrac}
\usepackage{mathtools}
\usepackage{tikz}
\usepackage{tikz-cd} %Use this for commutative diagrams
\usepackage{verbatim}
\usepackage{listings}
\usepackage[colorlinks=true,citecolor=blue,linkcolor=blue,urlcolor=blue]{hyperref}
\usepackage{color}
\usepackage{bm}
\usepackage{enumitem}
\setlist[itemize]{leftmargin=*}
\usepackage[T1]{fontenc}
\usepackage{multirow}
\usepackage{booktabs}

\definecolor{codegreen}{rgb}{0,0.6,0}
\definecolor{codegray}{rgb}{0.5,0.5,0.5}
\definecolor{codepurple}{rgb}{0.58,0,0.82}
\definecolor{backcolour}{rgb}{0.96,0.96,0.96}

\lstdefinestyle{mystyle}{
    backgroundcolor=\color{backcolour},   
    commentstyle=\color{codegreen},
    keywordstyle=\color{blue},
    numberstyle=\tiny\color{codegray},
    stringstyle=\color{codepurple},
    basicstyle=\ttfamily\footnotesize,
    breakatwhitespace=false,         
    breaklines=true,                 
    captionpos=b,                    
    keepspaces=true,                 
    numbers=left,                    
    numbersep=5pt,                  
    showspaces=false,                
    showstringspaces=false,
    showtabs=false,                  
    tabsize=2
}
\usetikzlibrary{arrows,calc,decorations,decorations.markings,fadings,positioning,patterns,shapes}
\tikzset{>=latex}
\tikzstyle{mypoint}=[inner sep=0pt,outer sep=0pt,minimum size=5pt,fill,circle]

\theoremstyle{definition}

\newtheorem*{theorem}{\textsc{\textbf{Theorem}}}          \newtheorem{theorem-N}[Thm]{\textsc{\textbf{Theorem}}}
\newtheorem*{lemma}{\textsc{\textbf{Lemma}}}              \newtheorem{lemma-N}[Lem]{\textsc{\textbf{Lemma}}}
\newtheorem*{corollary}{\textsc{\textbf{Corollary}}}      \newtheorem{corollary-N}[Cor]{\textsc{\textbf{Corollary}}}
 \newtheorem{proposition-N}[Prop]{\textsc{\textbf{Proposition}}}

\newtheorem*{definition}{\textsc{\textbf{Definition}}}    \newtheorem{definition-N}[def]{\textsc{\textbf{Definition}}}
\newtheorem*{remark}{\textsc{\textbf{Remark}}}            \newtheorem{remark-N}[rem]{\textsc{\textbf{Remark}}}
\newtheorem*{example}{\textsc{\textbf{Example}}}          \newtheorem{example-N}[ex]{\textsc{\textbf{Example}}}
      \newtheorem{observation-N}[Obs]{\textsc{\textbf{Observation}}}

\newtheorem{conjecture-N}[Conj]{\textsc{\textbf{Conjecture}}}

\newtheorem{algorithm-N}[Alg]{\textsc{\textbf{Algorithm}}}

\newenvironment{display}{\begin{center}\begin{tikzcd}}{\end{tikzcd}\end{center}}

\usepackage{etoolbox}
\newcommand{\define}[4]{\expandafter#1\csname#3#4\endcsname{#2{#4}}}

\forcsvlist{\define{\DeclareMathOperator}{}{}}{im,coker,rad,nil,Ann,Ass,codim,Spec,mSpec,diam,ord,Supp,supp,disc,Ob,vol,rank,Sym,Alt,Ind,,tr,spl,re, Jac, Pic, Div, gr}
\forcsvlist{\define{\newcommand}{\mathrm}{}}{Hom,Mor,id,GL,SL,PSL,USp,PGL,SO,SU,U,Mat,Ext,Tor,Res,Cor,Inf,End,Irr,Aut,Gal,Lie,lcm,sign,triv,diag,Map,op,ev,act,alg,sep,unr,nr,ab,N,T,ad}

\forcsvlist{\define{\newcommand}{\mathsf}{}}{Set,Grp,Ab,CRing,Mod,Vect,Cat,Top,PreSh,Sh,Sch,Nat,Fun,Diff,Alg,Rep,Aff}

\forcsvlist{\define{\newcommand}{\mathbf}{}}{Z,Q,R,C,F,G,A,B,D}

\newcommand{\qr}[2]{\genfrac{(}{)}{}{}{#1}{#2}}

\renewcommand{\sl}{\mathfrak{sl}_2}
\renewcommand{\O}{\mathcal{O}}
\renewcommand{\H}{\mathrm{H}}

\makeatletter
\newcommand*\bigcdot{\mathpalette\bigcdot@{.5}}
\newcommand*\bigcdot@[2]{\mathbin{\vcenter{\hbox{\scalebox{#2}{$\m@th#1\bullet$}}}}}
\makeatother

{\makeatletter
 \gdef\xxxmark{%
  \expandafter\ifx\csname @mpargs\endcsname\relax % in minipage?
     \expandafter\ifx\csname @captype\endcsname\relax % in figure/caption?
      \marginpar{xxx}% not in a caption or minipage, can use marginpar
     \else
      xxx % notice trailing space
     \fi
  \else
     xxx % notice trailing space
  \fi}
 \gdef\xxx{\@ifnextchar[\xxx@lab\xxx@nolab}
 \long\gdef\xxx@lab[#1]#2{\textbf{[\xxxmark #2 ---{\sc #1}]}}
 \long\gdef\xxx@nolab#1{\textbf{[\xxxmark #1]}}
}

\begin{document}

\title{Filtrations on block subalgebras of reduced universal enveloping algebras}

\author{Andrei Ionov}

\address{Department of Mathematics,
Massachusetts Institute of Technology,
77 Massachusetts Ave., Cambridge, MA 02139, United States; Department of Mathematics, National Research University Higher School of Economics, 6 Usacheva st., Moscow, Russian Federation, 119048}
\email{aionov@mit.edu}

\author{Dylan Pentland}

\address{Massachusetts Institute of Technology,
77 Massachusetts Ave., Cambridge, MA 02139, United States}
\email{dylanp@mit.edu}

\begin{abstract}
We study the interaction between the block decompositions of reduced universal enveloping algebras in positive characteristic, the PBW filtration, and the nilpotent cone. We provide two natural versions of the PBW filtration on the block subalgebra $A_\lambda$ of the restricted universal enveloping algebra $\mathcal{U}_\chi(\mathfrak{g})$ and show these are dual to each other. We also consider a shifted PBW filtration for which we relate the associated graded algebra to the algebra of functions on the Frobenius neighbourhood of $0$ in the nilpotent cone and the coinvariants algebra corresponding to $\lambda$. In the case of $\mathfrak{g}=\mathfrak{sl}_2(k)$ in characteristic $p>2$ we determine the associated graded algebras of these filtrations on block subalgebras of $\mathcal{U}_0(\mathfrak{sl}_2)$. We also apply this to determine the structure of the adjoint representation of $\mathcal{U}_0(\mathfrak{sl}_2)$.
\end{abstract}

\maketitle

\tableofcontents

\section{Introduction}

The key feature of the representation theory of algebraic groups and Lie algebras in characteristic $p>0$ is the linkage principle providing the block decomposition of the category of representations 
(\cite{curtis1966representation}, \cite{humphreys1978blocks}, \cite{andersen1980strong}, \cite{humphreys1998modular}, \cite{brownblocks}). In order to formulate it for a Lie algebra $\mathfrak{g}$ of an algebraic group $G$ one first decomposes the category of representations with respect to the $p$-character $\chi\in\mathfrak{g}^*$ and passes to the corresponding reduced universal enveloping algebra $\mathcal{U}_\chi(\mathfrak{g})$. We note here that via the differential map the case $\chi=0$ is related to the representation theory of algebraic group $G$ by the equivalence of categories $\Rep(G_1) \simeq \mathcal{U}_0(\mathfrak{g})-\Mod$, where $G_1$ is the Frobenius kernel of $G$ and, therefore, is of particular importance. Finally, for $\mathcal{U}_\chi(\mathfrak{g})$ the block decomposition could be understood as the direct sum decomposition of algebras 
\[\mathcal{U}_\chi(\mathfrak{g}) = \bigoplus_\lambda A_\lambda,\]
indexed by certain equivalence-classes of simple $\mathcal{U}_\chi(\mathfrak{g})$-modules. In the current paper we use the PBW filtration on $\mathcal{U}(\mathfrak{g})$ to study the block subalgebras  $A_\lambda$.

We study three filtrations on the block subalgebra $A_\lambda$. The first two are natural ways to put a PBW filtration on $A_\lambda$, by either pushing forward the PBW filtration on $\mathcal{U}_\chi(\mathfrak{g})$ by the map $\pi_\lambda$ projecting onto $A_\lambda$ or by intersection with the PBW filtration on $\mathcal{U}_\chi(\mathfrak{g})$. The shifted PBW filtration is obtained from the push forward filtration by additionally assigning the generators of the Harish-Chandra center degree $1$.

We realize associated graded algebra $\gr_{\mathrm{sh}} A_\lambda$ as a quotient of a tensor product $k[\mathcal{N}_p]\otimes_k C_\lambda$, where $\mathcal{N}_p$ is the Frobenius neighbourhood of $0$ inside the nilpotent cone and $C_\lambda$ is the coinvariants algebra corresponding to the block. For the trivial character this quotient is moreover $G$-equivariant. We also show that the first two PBW filtrations we put on $A_\lambda$ are dual to each other.

In the case of $G=\SL_2(k)$, we compute explicit central idempotents for the block decomposition of $\mathcal{U}_\chi(\mathfrak{g})$. In this case for $\chi=0$ or regular semisimple we are furthermore explicitly determine associated graded algebras of the filtrations as well as determining the structure of $\SL_2$-representations on them for $\chi=0$.
Finally, we use this to determine the structure of the adjoint representation of $\mathcal{U}_0(\mathfrak{sl}_2)$. Our result agrees with the main result of \cite{ostrik} for the small quantum group of $\sl$ at the $p$-th root of unity.

The key ingredient of our approach is understanding the structure of $G$-representation on the algebra of functions $k[\mathcal{N}]$ on the nilpotent cone, which in the case of $\chi=0$ can be exploited to understand block subalgebras  as $G$-modules. In the case of characteristic $0$ the answer were provided in \cite{kostant1963lie}. In the case of characteristic $p$ though the question is substantially less explored and in general only the existence of the good filtration on $k[\mathcal{N}]$ is known (\cite{kumar1999frobenius}, \cite{jantzenbook}, \S II Proposition 4.20). 

This approach to the block subalgebras could possibly be applied to understanding the center of the block subalgebra. The multiplicity of the trivial $\SL_2$-representation as a subrepresentation in $A_\lambda$ or $\gr A_\lambda$ provides an upper bound on the dimension of the center of $A_\lambda$. It would be interesting to understand if this method could be used in the general case to understand $\dim Z(A_\lambda)$. 

We note that $A_\lambda$ for a block containing a single simple projective module $L$, we have $A_\lambda \simeq \End(L)$. It was communicated to us by Pavel Etingof that even in the characteristic $0$ case the pushforward of the PBW filtration onto $\End(L)$ for simple modules is not well-understood, despite it being an important problem. 

\noindent \textbf{\textsc{Notation.}} Throughout, $k$ will denote an algebraically closed field of characteristic $p>0$. We let $X^{(1)} = X \times_k \Spec k_{\mathrm{Frob}}$ be the Frobenius twist of an affine variety over $k$, where $\Spec k_{\mathrm{Frob}}$ is a scheme $\Spec k$ with $k$-algebra structure on $k$ given via $\mathrm{Frob}: k \to k, x \mapsto x^p$. We equip this with a map $\mathrm{Fr}: X \to X^{(1)}$ arising from the universal property of the fibre product and the absolute Frobenius $X \to X$.

By an algebraic group $G$ we will mean a linear algebraic group, so that as a variety it is affine. We put $\mathfrak{g}$ for the Lie algebra of $G$.
We use $\mathcal{U}(\mathfrak{g})$ to denote the universal enveloping algebra, and $\mathcal{U}_\chi(\mathfrak{g})=\mathcal{U}(\mathfrak{g})/\langle x^p - x^{[p]}-\chi(x)^p\rangle$ to denote the reduced universal enveloping algebra for $\chi \in \mathfrak{g}^*$. 
There is an adjoint action of $G$ on $\mathcal{U}(\mathfrak{g})$ induced by the adjoint action of $G$ on $\mathfrak{g}$. For $\chi=0$, the map $\mathcal{U}(\mathfrak{g}) \to \mathcal{U}_0(\mathfrak{g})$ is $G$-equivariant, as we have $\mathrm{Ad}_g(x^{[p]}) = (\mathrm{Ad}_g(x))^{[p]}$. This yields a $G$-action on $\mathcal{U}_0(\mathfrak{g})$. 

We use $\Rep(G), \Rep(\mathfrak{g})$ for the categories of representations of $G$ and $\mathfrak{g}$ respectively. We put $\mathfrak{h}\subset\mathfrak{g}$ for a Cartan subalgebra, which we assume to be fixed, and $W$ for the corresponding Weyl group. We put $\bigcdot$ to denote the dot-action on the weight lattice and $\mathfrak{h}$: $w\bigcdot\lambda\colon=w(\lambda+\rho)-\rho$, with $\rho$ being the sum of fundamental weights.

For a graded vector space $V$, we write $V_d$ for the degree $d$ component. We set $V_{\le d} := \bigoplus_{d'\le d} V_{d'}$.

\section{Background}

This section will cover the background material on the block decomposition for $\mathcal{U}_\chi(\mathfrak{g})$. We will begin with general background, and at the end of the section discuss how these results present themselves in the case $G=\SL_2$. For further details we refer to \cite{jantzenbook}, a good summary of the results is also presented in \cite{ciappara2020lectures}.

\subsection{Structure of the center} \label{sect: center}

Let $G$ be a semisimple connected simply connected algebraic group. Assume that the characteristic $p$ of the field is `very good' as in \cite{centers}. Throughout the paper, we will need to understand the structure of the center of $\mathcal{U}(\mathfrak{g})$ and its image under the quotient map to $\mathcal{U}_\chi(\mathfrak{g})$. 
\begin{definition}
To describe the center $Z(\mathcal{U}(\mathfrak{g}))$, we will need two subalgebras of $\mathcal{U}(\mathfrak{g})$:
\begin{itemize}
    \item We set $Z_{\mathrm{Fr}} :=  k[x_i^p - x_i^{[p]}|i=1,\ldots,\dim\mathfrak{g}]$ to be the $p$-center, where the collection of elements $x_i\in \mathfrak{g}$ is a basis of the Lie algebra as a vector space.
    \item We set $Z_{\mathrm{HC}}:=\mathcal{U}(\mathfrak{g})^G$ to be the subalgebra of $G$-invariant elements in $\mathcal{U}(\mathfrak{g})$ under the adjoint action of $G$ on $\mathcal{U}(\mathfrak{g})$, which is called the Harish-Chandra center.
\end{itemize}
\end{definition}
Under our assumptions we may write the center down explicitly using the following theorem:
\begin{theorem}
We have
\[Z(\mathcal{U}(\mathfrak{g})) = Z_{\mathrm{Fr}} \otimes_{Z_{\mathrm{Fr}}^G} Z_{\mathrm{HC}}\]
where elements of $Z_{\mathrm{Fr}}^G$ are the $G$-invariant elements in the $p$-center.
\end{theorem}

\begin{proof}
See Theorem 2 of \cite{centers}.
\end{proof}

Furthermore, we may make further characterizations of $Z_{\mathrm{Fr}}^G$ and $Z_{\mathrm{HC}}$ - the Harish-Chandra isomorphism allows us to understand $G$-invariant elements through the Cartan subalgebra $\mathfrak{h}$.

\begin{theorem}
We have an isomorphism
\begin{display}\Theta_{\mathrm{HC}}: Z_{\mathrm{HC}} \ar{r}{\simeq} & \mathrm{S}(\mathfrak{h})^{W, \bigcdot}.\end{display}
Moreover, the center $Z_{\mathrm{HC}}$ is a polynomial algebra $k[c_1, \ldots, c_n]$ for a set of generators $c_1,\ldots, c_n$ ($n=\rank \mathfrak{g}$). Furthermore, upon restricting to $Z_{\mathrm{Fr}}^G = Z_{\mathrm{Fr}} \cap Z_{\mathrm{HC}}$ we get an isomorphism
\begin{display}
Z_{\mathrm{Fr}}^G \ar{r}{\simeq} & \mathrm{S}(\mathfrak{h}^{(1)})^{W,\bigcdot}.
\end{display}
The map $\mathrm{S}(\mathfrak{h}^{(1)})^{W,\bigcdot}\to \mathrm{S}(\mathfrak{h})^{W,\bigcdot}$ is  the Artin-Schreier map.
\end{theorem}

\begin{proof}
The first claim is the main result of \cite{centeriso}. The second claim is Lemma 4 in \cite{centers}.
\end{proof}

When considering the image of $Z(\mathcal{U}(\mathfrak{g}))$ in $\mathcal{U}_\chi(\mathfrak{g})$, we may use these results to understand this as well.

\begin{lemma}
The image of the $Z(\mathcal{U}(\mathfrak{g}))$ in $\mathcal{U}_\chi(\mathfrak{g})$ is given by
\[k_\chi \otimes_{Z_{\mathrm{Fr}}^G} Z_{\mathrm{HC}},\]
where $k_\chi$ is the 1-dimensional algebra over $Z_{\mathrm{Fr}}$ induced by the character $\chi$.
\end{lemma}

\begin{proof}
This follows from applying the quotient map to $Z(\mathcal{U}(\mathfrak{g})) = Z_{\mathrm{Fr}} \otimes_{Z_{\mathrm{Fr}}^G} Z_{\mathrm{HC}}$.
\end{proof}

\subsection{The block decomposition}
\label{sect: block_decomp}

Suppose $G$ is a reductive algebraic group. We may decompose $\Rep(G)$ as a direct sum of smaller categories, which creates the block decomposition $\mathcal{U}_\chi(\mathfrak{g})=\bigoplus_\lambda A_\lambda$ that we study. Before stating this decomposition, we need to define the notion of a Serre subcategory.

\begin{definition}
Let $\mathcal{A}$ be an abelian category. A Serre subcategory $\mathcal{S}\subseteq \mathcal{A}$ is a non-empty full subcategory so that for any exact sequence
\begin{display}
0 \ar{r} & A \ar{r} & B \ar{r} & C \ar{r} & 0
\end{display}
where $A,C \in \mathcal{S}$ if and only if $B\in \mathcal{S}$.
\end{definition}

\begin{proposition-N}
Let $\mathfrak{X}$ denote the weight lattice of $G$, $\mathfrak{X}_+\subset \mathfrak{X}$ be the subset of positive weights, $W^{\mathrm{aff}}$ the affine Weyl group, and $\bigcdot_p$ the $p$-dilated dot action of $W^{\mathrm{aff}}$. We have a decomposition
\[\Rep(G) = \bigoplus_{\lambda \in \mathfrak{X}/(W^{\mathrm{aff}},\bigcdot_p)} \Rep(G)_\lambda.\]
Here, $\Rep(G)_\lambda$ is the Serre subcategory generated by simple modules $L_{\lambda'}$ for $\lambda'$ in $(W^{\mathrm{aff}} \bigcdot_p \lambda) \cap \mathfrak{X}_+$.
\label{alg_group_blocks}
\end{proposition-N}

\begin{proof}
This is shown in \cite{donkin1980blocks}, and discussed in \cite{ciappara2020lectures}.
\end{proof}

We note that when we refer to $\Rep(G)_\lambda$ we fix a fundamental domain for $\mathfrak{X}/(W^{\mathrm{aff}}, \bigcdot_p)$. This applies to all other block decompositions as well. When explicitly needed, we will give a description of this fundamental domain. 

The similar results hold for $\Rep(\mathfrak{g})$.

\begin{proposition-N}
Let $\mathfrak{X}$ denote the weight lattice of $G$, $W$ the finite Weyl group, and $\bigcdot$ the dot action of $W$. We have a decomposition
\[\Rep(\mathfrak{g}) = \bigoplus_{\lambda \in  \mathfrak{h}^*/(W,\bigcdot)} \Rep(\mathfrak{g})_\lambda.\]
Here, $\Rep(\mathfrak{g})_\lambda$ is the Serre subcategory generated by the simple modules $L_{\lambda'}$ for $\lambda'$ in $(W \bigcdot \lambda) \subset\mathfrak{h}^*$.
\label{liealgblocks}
\end{proposition-N}

In the case $\lambda = w \cdot \mu$ for $w\in W$, we say that $\lambda, \mu$ are $W$-linked. Letting $\chi \in \mathfrak{g}^*$ be a character of $\mathfrak{g}$, we may pass to the reduced universal enveloping algebra \[\mathcal{U}_\chi(\mathfrak{g}) = \mathcal{U}(\mathfrak{g})/\langle x^p-x^{[p]} - \chi(x)^p \rangle.\]
Through the universal property of $\mathcal{U}(\mathfrak{g})$, we have a bijection
\begin{display}
\{\mathfrak{g} - \mathrm{modules}\} \ar[r,leftrightarrow] & \{\mathcal{U}(\mathfrak{g}) - \mathrm{modules}\}.
\end{display}
This restricts to $\mathcal{U}_\chi(\mathfrak{g})$-modules when we consider $\mathfrak{g}$-modules with $p$-character $\chi$. 

Restricting the result of Proposition \ref{liealgblocks} to the subcategory $\Rep(\mathfrak{g})_\chi$ of $\mathfrak{g}$-modules with a given character $p$-character $\chi$, we obtain a block decomposition of $\Rep(\mathfrak{g})_\chi$. The  simple modules in $\Rep(\mathfrak{g})_\chi$ are constructed as the quotients of the so-called baby Verma modules given by $\Delta_{\chi,\lambda} := \mathcal{U}_\chi(\mathfrak{g}) \otimes_{\mathcal{U}_\chi(\mathfrak{b}^+)} k_\lambda$. It is known that all simple $\mathcal{U}_\chi(\mathfrak{g})$-modules arise in this way, see \cite{humphreys1998modular}.

We have a block decomposition in this case, which is simply an algebraic fact. One may define a linkage relation among principal indecomposable modules for $\mathcal{U}_\chi(\mathfrak{g})$, as in \cite{curtis1966representation} \S 55.

\begin{theorem-N}
We have a decomposition
\[\Rep(\mathfrak{g})_\chi = \bigoplus_{\lambda} A_\lambda - \Mod\]
where $A_\lambda$ is the two-sided ideal of $\mathcal{U}_\chi(\mathfrak{g})$ equal to the sum of all principal indecomposable modules belonging to an equivalence class under the linkage relation.
\label{liealgcharblocks}
\end{theorem-N}

\begin{proof}
See \cite{curtis1966representation}, Theorem 55.2 in \S 55.
\end{proof}

The blocks can again be viewed as Serre subcategories for corresponding simple modules. In the case of $G=\SL_2$, we know this decomposition explicitly because there is a complete understanding of the baby Verma modules and simple modules for any character $\chi$.

A more precise result is an explicit description of the blocks in terms of $W$-linkage classes. We may explicitly characterize the blocks by the associated baby Verma modules. We assume $\mathfrak{g}$ is semisimple, and define a modified weight lattice $\mathfrak{X}_\chi$ as in \cite{jantzen1998representations} \S 6.2 via
\[\mathfrak{X}_\chi = \{ \lambda \in \mathfrak{h}^* : \lambda(h)^p-\lambda(h^{[p]}) = \chi(h)^p \text{ for all } h\in \mathfrak{h} \}.\]
We have $|\mathfrak{X}_\chi| = p^{\dim \mathfrak{h}}$. This will index all of the baby Verma modules $\Delta_{\chi, \lambda}$.

\begin{theorem-N}
For $\chi \in \mathfrak{g}^*$, the blocks of $\mathcal{U}_\chi(\mathfrak{g})$ are in natural bijection with the $W$-linkage classes of weights in $\mathfrak{X}_\chi$.
\label{blocks_and_W-linkage}
\end{theorem-N}

\begin{proof}
See \cite{brownblocks} \S 3.18.
\end{proof}

We note that the Harish-Chandra center of $\mathfrak{g}$ and, in particular, the elements $c_1,\ldots, c_n$, act by a constant on the simple modules within a given block. For $\sl$ this action uniquely determines the block and thus we will sometimes refer to the block by the constant by which the generator of the Harish-Chandra center acts on the simples in the block.

It is worth noting that in characteristic $p$, the differentiation functor
\begin{display}
D: \Rep(G) \ar{r} & \Rep(\mathfrak{g})
\end{display}
induces an equivalence of categories between $\Rep(G_1)$ and $\mathcal{U}_0(\mathfrak{g})$-modules, so in this case Theorem \ref{blocks_and_W-linkage} is an immediate consequence of a block decomposition on $\Rep(G_1)$. Here, $G_1$ is the kernel of the Frobenius twist $\mathrm{Fr}: G \to G^{(1)}$. This illustrates the connection between the two block decompositions.

Given this block decomposition, we may decompose the algebra $\mathcal{U}_\chi(\mathfrak{g})$ into blocks associated to each Serre subcategory. In particular, we may write
\[\mathcal{U}_\chi(\mathfrak{g}) = \bigoplus_{\lambda \in \mathfrak{X}_\chi/(W,\bigcdot_p)} A_\lambda\]
where nonzero element $x\in A_\lambda$ if it acts nontrivially on $\Rep(\mathfrak{g})_{\chi,\lambda}$ and trivially on all the other subcategories in the decomposition of Theorem \ref{blocks_and_W-linkage}. 

For any decomposition of an algebra into direct summands, there exists a corresponding collection of central idempotents. In particular, the decomposition $\mathcal{U}_\chi(\mathfrak{g}) = \bigoplus_\lambda A_\lambda$ can be written as
\[\mathcal{U}_\chi(\mathfrak{g}) = \bigoplus_\lambda \pi_\lambda \mathcal{U}_\chi(\mathfrak{g}),\]
where $\pi_\lambda$ are central idempotents. Note that due to the result of \cite{centralidempotents}, the central idempotents of $\mathcal{U}_\chi(\mathfrak{g})$ lie in the image of the center of $\mathcal{U}(\mathfrak{g})$. This allows to interpret the projection maps onto each block as multiplication by $\pi_\lambda$.

\begin{remark}
By abuse of notation we will refer to the projection maps onto blocks as well as the central idempotents in the above decomposition by $\pi_\lambda$.
\end{remark}

Recall from the introduction that there is a natural $G$-action on $\mathcal{U}_0(\mathfrak{g})$ induced by the adjoint action of $G$. Since the central idempotents $\pi_\lambda$ are $G$-invariant, the corresponding projections are $G$-equivariant and the adjoint action preserves the subalgebras $A_\lambda$.  

\subsection{Representation theory of \texorpdfstring{$\SL_2$}{SL2} and \texorpdfstring{$\sl$}{sl2}}

In this subsection, we briefly review the representation theory of $\SL_2$ and $\sl$ in characteristic $p$. A good reference for this is \cite{sl2reps}. This is also summarized (on a less detailed level) in \cite{humphreys1998modular} and \cite{jantzen1998representations}. Throughout this subsection, we will need $p>2$ for these results to hold.

The flag variety of $\SL_2$ is $\mathbf{P}^1$. This group has weight lattice $\mathfrak{X}\simeq \Z$ and $\rho=1$. We construct representations of $\SL_2$ known as Weyl modules - these are defined as $\nabla_\lambda = \H^0(\mathbf{P}^1, \O(\lambda))$. We can identify this with
\[\nabla_\lambda = k[x,y]_{\lambda} = k x^\lambda \oplus kx^{\lambda-1} y \oplus \ldots \oplus k x y^{\lambda-1} \oplus k y^\lambda\]
where $x,y$ are homogeneous coordinates for $\mathbf{P}^1$ and the subscript denotes the degree $\lambda$ part of the grading on $k[x,y]$. We note that $\nabla_\lambda$ are indexed by $\mathfrak{X}_+$ and the Weyl modules are defined for  a general reductive algebraic group in a similar fashion.

\begin{lemma-N}
For $\lambda < p$, the module $\nabla_\lambda$ is simple. For $\lambda \ge p$, it is no longer simple, with simple socle $L_\lambda$. We have 
\[L_{\mu p+\lambda} \simeq L_\lambda \otimes\mathrm{Fr}^* L_\mu.\]
As a submodule of $\nabla_\lambda$, we may embed $L_\lambda$ as $G\cdot x^\lambda \subseteq \nabla_\lambda$.
\label{frob_simple_twist}
\end{lemma-N}

% Incorrect statement

% We will also need a description of projective covers of $L_\lambda$ for $\lambda < p$.

% \begin{lemma-N}
% The simple module $L_{p-1}=\nabla_{p-1}$ is projective and known as the Steinberg representation. For $\lambda<p-1$ the projective cover $P_\lambda$ of $L_\lambda$ is a self-dual $2p$-dimensional representation. In particular, it is also an injective hull of $L_\lambda$.
% \label{projcovers-group}
% \end{lemma-N}

% \begin{proof}
% See for example \cite{humphreys1973projective}.
% \end{proof}

Next, we discuss the representations of $\sl$. We first need to choose a basis, which will be fixed for the rest of the paper.

\begin{definition}
Let $e=\begin{pmatrix} 0 & 1 \\ 0 & 0 \end{pmatrix}$, $f=\begin{pmatrix} 0 & 0 \\ 1 & 0 \end{pmatrix}$, and $h=\begin{pmatrix} 1 & 0 \\ 0 & -1 \end{pmatrix}$. These form a basis for $\sl$, and have the relations
\[[e,f]=h, [h,f]=-2f, [h,e]=2e.\]
\end{definition}

Because we assume $p>2$, the trace form $4\mathrm{tr}(xy)$ is non-degenerate and hence we have $\sl \simeq \sl^*$. Thus, when classifying characters we may consider the associated elements in $\sl$. We may divide them into three categories, up to conjugacy by $\SL_2$:

\begin{lemma-N}
We may divide the characters of $\sl$ into three cases up to conjugacy:
\begin{enumerate}
    \item $\chi=0$. 
    \item $\chi\neq 0$ but nilpotent. We pick $\chi = e$ as a representative.
    \item $\chi$ regular. We pick $ah/2$ with $a\neq 0$ as a representative.
\end{enumerate}
\label{charactercases}
\end{lemma-N}

For each of these cases, we have a corresponding classification of simple $\mathcal{U}_\chi(\sl)$ modules, or simple $\sl$-modules with $p$-character $\chi$. To distinguish from the notation for $\SL_2$-modules, we add the subscript $\chi$ to denote which character we are using. We state with these the grouping into different blocks following Theorem \ref{blocks_and_W-linkage}.

\begin{lemma-N}
We have three cases for the simple $\mathcal{U}_\chi(\sl)$ modules corresponding to Lemma \ref{charactercases}. 
\begin{enumerate}
    \item When $\chi=0$, there are $p$ nonisomorphic simple modules $L_{0,\lambda}$ for $0\le \lambda < p$ which are quotients of the baby Verma modules via the exact sequences
    \begin{display}
    0 \ar{r} & L_{0,-\lambda-2} \ar{r} & \Delta_{0, \lambda} \ar{r} & L_{0,\lambda} \ar{r} & 0.
    \end{display}
    We group $L_{0,\lambda}, L_{0,-\lambda -2}$ together in the block decomposition for $\mathcal{U}_0(\sl)$, meaning we have a single block at with a single simple module at $\lambda=p-1$ called the Steinberg module and $(p-1)/2$ blocks with two simple modules.
    \item When $\chi=e$, all of the baby Verma modules are simple. We have $(p+1)/2$ non-isomorphic baby Verma modules. We have $\Delta_{e,\lambda} \simeq \Delta_{e,-\lambda-2}$ over $\lambda \in \F_p$. Each block has a single simple module.
    \item When $\chi = a h/2$ is regular, we have all $\Delta_{a h/2,\lambda}$ simple over $\lambda$ roots of $x^p-x=a$ and they are non-isomorphic. In the block decomposition, each block has a single simple module.
\end{enumerate}
\label{sl2_block_decomp_cases}
\end{lemma-N}

In our computations, we will also need some understanding of the projective covers of $L_{\chi,\lambda}$. 

\begin{lemma-N} We have three cases.
\begin{enumerate}
    \item $\chi=0$: Each simple $L_{0,\lambda}$ for $\lambda \in \F_p$ (as in Lemma \ref{sl2_block_decomp_cases}) has a projective cover. The simple module $L_{0,p-1}=P_{0,p-1}$ is the unique simple projective, and every other $L_{0,\lambda}$ has a dimension $2p$ indecomposable projective cover $P_{0,\lambda}$ with composition factors $L_{0,\lambda}, L_{0,p-\lambda-2}$. 
    \item $\chi=e$: Each $\Delta_{e,\lambda}$ has an indecomposable projective cover $P_{e,\lambda}$ of dimension $2p$ which is a self-extension.
    \item $\chi = a h/2$: Each simple module is projective.
\end{enumerate}

\label{projective_covers}
\end{lemma-N}
\begin{proof}
See \S 13 of \cite{humphreys1998modular}.
\end{proof}

The representations of $\SL_2$ and $\sl$ are related by differentiation functor in the following way.

\begin{lemma-N}
We have $D(L_\lambda) \simeq L_{0,\lambda}$ for $\lambda < p$, and $D(\mathrm{Fr}^* L_\mu)$ is a trivial representation so that in general
\[D(L_{\mu p + \lambda}) \simeq L_{0,\lambda}^{\oplus\mu+1}.\]

\label{diff_simples}
\end{lemma-N}

\begin{proof}
See \cite{sl2reps} \S 2.4.
\end{proof}

Using the following lemma, we can determine bases for $\mathcal{U}_\chi(\sl)$ and write down the kernel of $\mathcal{U}(\mathfrak{g}) \to \mathcal{U}_\chi(\mathfrak{g})$.

\begin{lemma-N}
For any character $\chi$, we have $e^i f^j h^k$ for $0 \le i,j,k < p$ as a basis for $\mathcal{U}_\chi(\mathfrak{g})$. Depending on the character, we have three possibilities for the ideal $\ker(\mathcal{U}(\mathfrak{g}) \to \mathcal{U}_\chi(\mathfrak{g}))$.
\begin{enumerate}
    \item For $\chi=0$, it is $\langle e^p, f^p, h^p-h \rangle$.
    \item For $\chi=e$, it is $\langle e^p, f^p-1, h^p-h \rangle$.
    \item For $\chi=\frac{a h}{2}$, it is $\langle e^p, f^p, h^p-h-a \rangle$.
\end{enumerate}
\label{restricted_basis}
\end{lemma-N}

Finally, we discuss the center of $\mathcal{U}(\sl)$. As discussed above, this is generated by $Z_{\mathrm{Fr}}$ and $Z_{\mathrm{HC}}$. The lemma above essentially tells us the form of $Z_{\mathrm{Fr}}$. As for $Z_{\mathrm{HC}}$, we observe that for $\sl$ the Cartan subalgebra $\mathfrak{h}$ is simply the span of $h$. Then we have $\mathrm{S}(\mathfrak{h}) \simeq k[h]$, and thus
\[Z_{\mathrm{HC}} \simeq k[h]^{(W,\bigcdot)}.\]
The finite Weyl group for $\sl$ is isomorphic to $\Z/2\Z$, and the dot action sends $h \mapsto -h-2$. Thus, $Z_{\mathrm{HC}}$ is isomorphic to $k[(h-1)^2]$, and has a single generator. For our convenience we will denote $h-1$ by $h'$.

\begin{lemma-N}
Under the Harish-Chandra isomorphism for $\sl$, the element $c = (h-1)^2 + 4ef \in Z_{\mathrm{HC}}$ maps to $h'^2$ in $k[h'^2]$.
\label{HCcenter}
\end{lemma-N}

\noindent This result tells us that $Z_{\mathrm{HC}}$ for $\sl$ is $k[c] \subset \mathcal{U}(\sl)$. This is not the usual quadratic Casimir element $\Omega = ef + fe + h^2/2$ for $\sl$, but rather $2\Omega+1$. For the convenience of notations and computations we shift it so that it vanishes at the fixed point of the $\bigcdot$-action.

\section{Filtrations on block subalgebras}

\subsection{Filtrations of interest} \label{sect: filtration_def}
We will be discussing three filtrations of interest in this section.  First, recall the definition of the PBW filtration.

\begin{definition}
The Poincar\'{e}-Birkhoff-Witt (PBW) filtration on $\mathcal{U}(\mathfrak{g})$ is the filtration $F_0 \subset F_1 \subset F_2 \subset \ldots$ with $F_d$ given by the degree $\le d$ elements with the degree induced from the tensor algebra.
\end{definition}

If we choose a basis $x_i$ of $\mathfrak{g}$ then we have monomial basis $\{\prod_i x_i^{e_i}\}$ of $\mathcal{U}(\mathfrak{g})$ as a vector space over $k$ by the PBW theorem. Then $F_d$ is the subspace of $\mathcal{U}(\mathfrak{g})$ spanned by elements $\prod_i x_i^{e_i}$ with  $\sum_i e_i \le d$. The filtration is compatible with multiplication in a sense that $F_i\cdot F_j\subset F_{i+j}$ and is uniquely characterized by this property together with $F_0=k\cdot 1, F_1=k\cdot 1+\mathfrak{g}\subset \mathcal{U}(\mathfrak{g})$. It is also preserved by the adjoint action of $G$ on $\mathcal{U}(\mathfrak{g})$. Finally, for the associated graded algebra of this filtration we have an isomorphism $\gr  \mathcal{U}(\mathfrak{g})\simeq S(\mathfrak{g})$.

We obtain a filtration on $\mathcal{U}_\chi(\mathfrak{g})$ by applying the quotient map to each vector space in the filtration. From this point onward, we will denote this PBW filtration on $\mathcal{U}_\chi(\mathfrak{g})$ by $V_d$. The associated graded algebra of this filtration is isomorphic to $k[\mathfrak{g}_p]$, where $\mathfrak{g}_p$ is the Frobenius neighbourhood of $0$ inside $\mathfrak{g}$, which is defined to be the preimage of $0$ under the Frobenius map $\mathrm{Fr}: \mathfrak{g}\to\mathfrak{g}^{(1)}$. Inside $\mathfrak{g}$ it is given by the ideal generated by the elements $x_i^p$, where $x_i$ runs through a basis of $\mathfrak{g}$.

\begin{definition-N} 
There are two natural ways we put a filtration on a block $A_\lambda$:
\begin{enumerate}
    \item The pushforward PBW filtration is given by $V_i^{\mathrm{pf}} := \pi_\lambda(V_i)$.
    \item The intersection PBW filtration is given by $V_i^{\mathrm{int}} := V_i \cap A_\lambda$.
\end{enumerate}
\label{pbw_on_blocks_def}
\end{definition-N}

We note that both filtrations are compatible with multiplication and the pushforward filtration is uniquely characterized by this together with $V_0^{\mathrm{pf}}=k\cdot \pi_\lambda$ and $V_1^{\mathrm{pf}}=k\cdot \pi_\lambda +\pi_\lambda \mathfrak{g}$. 

The last filtration we consider is a shifted PBW filtration.

\begin{definition-N}
 Define a filtration $V_i^{\mathrm{sh}}$ on the block $A_\lambda$ inductively via
\[V_i^{\mathrm{sh}} := \pi_\lambda(V_d) + V_{i-1}^{\mathrm{sh}}+ \sum_{j} c_j V_{i-1}^{\mathrm{sh}},\]
where the $c_j$ are as in \S \ref{sect: center}.
\label{shifted_pbw_def}
\end{definition-N}

This last filtration is also compatible with multiplication and is uniquely characterized by this together with $V_0^{\mathrm{sh}}=k\cdot \pi_\lambda, V_1^{\mathrm{sh}}=k\cdot \pi_\lambda +\pi_\lambda \mathfrak{g}+\mathrm{span}_j(\pi_\lambda c_j)$. 

The associated graded algebras for these on $A_\lambda$ are denoted by $\gr_{\mathrm{pf}} A_\lambda, \gr_{\mathrm{int}} A_\lambda, \gr_{\mathrm{sh}} A_\lambda$ respectively. Finally, note that for $\chi=0$ all three filtrations as well as $V_i$ are preserved by the $G$-actions.

\subsection{The shifted PBW filtration and the nilpotent cone}
\label{sect: shifted_PBW}

In this subsection we study the relation between the nilpotent cone $\mathcal{N}$ and the associated graded algebra of $A_\lambda$ with respect to the shifted PBW filtration $V_i^{\mathrm{sh}}$. 

Throughout this subsection, we let $G$ be a semisimple connected simply connected linear algebraic group with Lie algebra $\mathfrak{g}$. We additionally impose the following assumptions on $G$:
\begin{enumerate}
    \item The derived group $\mathcal{D}G$ of $G$ is simply connected.
    \item $p$ is odd and a good prime for $G$ (as in \S \ref{sect: center})
    \item The Killing form on $\mathfrak{g}$ is non-degenerate.
\end{enumerate}
Recall that $\mathcal{D}G$ is the intersection of all normal algebraic subgroups $N \le G$ so that $G/N$ is commutative. It can be viewed as the algebraic group analogue to the commutator subgroup $[G,G]$. With these conditions on $G$, we may use the relevant results for $G$ - we note that for $G=\SL_2$, all these conditions apply, so any theorems of this section may be used for $\sl$ in characteristic $p>2$. 

\begin{definition}
By abuse of notation we identify generators $c_i$ of the Harish-Chandra center with their image in the associated graded algebra $S(\mathfrak{g})$. We define the nilpotent cone $\mathcal{N}$ to be a closed subscheme $\Spec \mathrm{S}(\mathfrak{g})/\langle c_1, \ldots, c_n \rangle$ of $\mathfrak{g}$ considered as an affine space. 
\label{nilconegens}
\end{definition}

Under our assumptions on $G$ the set of closed points of $\mathcal{N}$ is the set of nilpotent elements of $\mathfrak{g}$ (see for example Theorem 2.6, \cite{friedlander1986cohomology}). Note also, that the adjoint $G$-action preserves $\mathcal{N}$.

As we work in characteristic $p>0$, we have a morphism
\begin{display}
\mathrm{Fr}: \mathcal{N} \ar{r} & \mathcal{N}^{(1)}.
\end{display}
We define the Frobenius neighborhood of $0$ in $\mathcal{N}$, denoted by $\mathcal{N}_p$, to be the preimage of a point $0\in \mathcal{N}^{(1)}$ under this map. This can be written as $\mathcal{N}_p \simeq \Spec k[\mathfrak{g}_p]/\langle c_1, \ldots, c_n \rangle$.

For this subsection, we fix a block $A_\lambda$ throughout. Let $\alpha_i$ be the constant by which the generator of the center $c_i$ acts on the simple modules associated to the block $A_\lambda$.

% \begin{definition}
% A character $\chi \in \mathfrak{g}^*$ is regular if $\dim \mathfrak{g}_\chi := \{ x\in \mathfrak{g} : \chi([x,\mathfrak{g}])=0 \}$ has minimal dimension $\mathrm{rank}(\mathfrak{g})$.
% \end{definition}

We put $C_\lambda$ for the image of the center of $\mathcal{U}(\mathfrak{g})$ in $A_\lambda$, which is known as the algebra of coinvariants. 

For a root $\mu$ of $G$ let $\mathfrak{g}_\mu$ be a corresponding root space. For a positive root $\mu$ of $G$ we let $h_\mu\in [\mathfrak{g}_\mu, \mathfrak{g}_{-\mu}]$ to be the unique element of $\mathfrak{h}$ so that $\mu(h_\mu)=2$ (recall $p>2$). We now have the following lemma, which is a combination of results from \cite{centers} and \cite{brownblocks}.

\begin{lemma-N}
The algebra $C_\lambda$ is a local ring. Moreover,
\begin{enumerate}
    \item For $\chi$ nilpotent, assigning degree $1$ to the elements $\pi_\lambda(c_i-\alpha_i)$ provides a grading on $C_\lambda$.
    \item We have $C_\lambda= k$ if and only if we have $\chi, \lambda$, such that $(\lambda + \rho)(h_\mu) \not \in \F_p^\times$ for all simple roots $\mu$. In particular, we will endow it with the trivial grading.
\end{enumerate}
\label{harishchandrakernel}
\end{lemma-N}

\begin{proof}
Let $\mathfrak{X}$ denote the weight lattice of $G$. Let $W(\lambda)$ be the stabilizer of $\lambda$ under the action of the finite Weyl group $W$, and $W(\lambda+\mathfrak{X})$ the stabilizer of $\lambda+\mathfrak{X}$, consisting of $w\in W$ so $w\lambda - \lambda \in \mathfrak{X}$. Given $\chi \in \mathfrak{g}^*$, $\chi^p$ is a character of the Frobenius twist $\mathfrak{g}^{*(1)}$. As discussed in \cite{centers}, this $p$th power has a Jordan decomposition $\chi^p = \chi_s + \chi_n$. In particular, by the result of \cite{centers} and the discussion of \S 3.8 in \cite{brownblocks} we have an isomorphism
\[C_\lambda\simeq \mathrm{S}(\mathfrak{h}^{(1)})^{W(\lambda)} \otimes_{\mathrm{S}(\mathfrak{h}^{(1)})^{W(\lambda+\mathfrak{X})}} k_{\chi_s}.\]
As $W(\lambda+\mathfrak{X})$ fixes $\chi_s$ this will be a local ring. Now further suppose $\chi$ is nilpotent. This identification shows that the algebra of coinvariants is graded because $\chi_s=0$ in this case. Explicitly identifying generators, we see that the grading is the degree grading in $c_i-\alpha_i$.

The second case follows from the theorem of \S 3.10 of \cite{brownblocks}, since the condition of part (2) of the theorem is equivalent to the algebra of coinvariants being $k$.
\end{proof}

\begin{remark}
For a regular semisimple character each block satisfies the condition of case (2) of the Lemma. The singular block (stabilized by the $W$-action) satisfies the condition of case (2) of the Lemma as well.
\end{remark}

In either case, the above Lemma implies that the inclusion $C_\lambda \to A_\lambda$ after passing to associated graded algebra becomes an injective map $C_\lambda \to \gr_{\mathrm{sh}} A_\lambda$ of graded algebras. In the situation of case (2) we have $V_i^{\mathrm{sh}}=V_i^{\mathrm{pf}}$, since the coinvariant algebra is trivial and so $\pi_\lambda c_j$ is a scalar element of $A_\lambda$.

Since we have $\pi_\lambda(V_i)\subset V_i^{\mathrm{sh}}$, there is a well defined map $\gr \mathcal{U}_\chi(\mathfrak{g})\to \gr_{\mathrm{sh}} A_\lambda$. Note that, since $G$ is semisimple, the images of the central elements $c_i$ have smaller degree with respect to the shifted filtration then with respect to pushforward filtration and, thus, the image of the ideal $\langle c_1, \ldots, c_n \rangle$ in $\gr \mathcal{U}_\chi(\mathfrak{g})$ lies in the kernel of the map. It follows that the above map quotients through $k[\mathcal{N}_p] \simeq k[\mathfrak{g}_p]/\langle c_1, \ldots, c_n \rangle$ since $k[\mathfrak{g}_p] \simeq \gr \mathcal{U}_\chi(\mathfrak{g})$. We then obtain a map
\begin{display}
k[\mathcal{N}_p] \ar{r} & \gr_{\mathrm{sh}} A_\lambda.
\end{display}
Furthermore, recall from the end of \S \ref{sect: block_decomp} that if $\chi=0$ there are $G$-actions on $\mathcal{U}_\chi(\mathfrak{g})$ and $A_\lambda$ and  the map $\mathcal{U}_\chi(\mathfrak{g})\to A_\lambda$ is $G$-equivariant and then so is the map $k[\mathcal{N}_p]\to \gr_{\mathrm{sh}} A_\lambda$.

Putting these maps together, we obtain a map from $k[\mathcal{N}_p] \otimes_k C_\lambda$ to $\gr_{\mathrm{sh}} A_\lambda$. We may then show the following theorem.

\begin{theorem-N}
Assume we are in either case (1) or (2) of Lemma \ref{harishchandrakernel}. We have a well defined surjective map
\begin{display}
k[\mathcal{N}_p] \otimes_k C_\lambda \ar[two heads]{r}& \gr_{\mathrm{sh}} A_\lambda,
\end{display}
of graded algebras. In particular, in case (2) of the above Lemma, we get a surjective map
\begin{display}
k[\mathcal{N}_p] \ar[two heads]{r}& \gr_{\mathrm{sh}} A_\lambda = \gr_{\mathrm{pf}} A_\lambda.
\end{display}
Furthermore, for $\chi=0$ the map is $G$-equivariant, with the $G$-action on $C_\lambda$ defined to be trivial.
\label{subschemestruct}
\end{theorem-N}

\begin{proof}
It remains to prove the surjectivity. The algebra $\gr_{\mathrm{sh}} A_\lambda$ is generated by the images of $c_j$ and $\mathfrak{g}\subset\mathcal{U}(\mathfrak{g})$ in it. The first ones lie in the image of $C_\lambda$ and the second ones lie in the image of $k[\mathcal{N}_p]$ as degrees of $c_j$ are bigger then $1$ and so $\langle c_1, \ldots, c_n \rangle\cap\mathfrak{g}=0$.
\end{proof}

Given that $C_\lambda$ is by definition a quotient of a polynomial $k$-algebra $Z_{\mathrm{HC}} \simeq \mathrm{S}(\mathfrak{h})^{W,\bigcdot}$, we may interpret its spectrum as a subscheme $X_\lambda$ of the affine space $\mathfrak{h}/W$. Given that $\mathcal{N}$ and hence $\mathcal{N}_p$ are affine schemes, we can interpret the tensor product as a fibre product of these schemes. The surjection onto the associated graded algebra tells us that 
\[\Spec(\gr_{\mathrm{sh}} A_\lambda) \subset \mathcal{N}_p \times_k X_\lambda\] 
can be realized as a closed subscheme of $\mathcal{N}_p \times_k X_\lambda$. Note that for case (2) of Lemma \ref{harishchandrakernel} $X_\lambda$ is a point and for a nilpotent character it is the spectrum of a local ring, so it has a unique closed point.

Let $W^\lambda$ be the set of minimal coset representatives for the subgroup $W(\lambda)$ (the stabilizer of $\lambda$). By \cite{brownblocks} \S 3.19, we, in addition, know the Hilbert series for $C_\lambda$, namely
\[H(X_\lambda,t) = \sum_{w\in W^{\lambda}} t^{\ell(w)}.\] In particular, we have $\dim C_\lambda = [W(\lambda+\mathfrak{X}) : W(\lambda)]$.

In the case of $\mathfrak{g}=\sl$ the scheme $X_\lambda$ is equal to $\Spec k$ or $\Spec k[\varepsilon]/\varepsilon^2$, with $\mathcal{N}_p\times_k X_\lambda$ being either $\mathcal{N}_p$ or a first order thickening of it. Geometrically, we aim to explicitly describe the ideal sheaves of $\gr A_\lambda$ as subschemes in the case of $\chi=0$. This is the same as determining the kernel of the map in Theorem \ref{subschemestruct}, which in the $\chi=0$ case is $G$-invariant. Once we decompose the ring of functions on the nilpotent cone into $G$-modules we may compute $k[\mathcal{N}_p] \otimes_k C_\lambda$ as a $G$-module.

\subsection{Duality between pushforward and intersection PBW filtrations}

In this section we show that the filtrations $V_i^{\mathrm{pf}}$ and $V_i^{\mathrm{int}}$ have a simple relationship, allowing us to only focus on the pushforward filtration.

The precise relation between these filtrations requires us to provide a non-degenerate associative bilinear form
\[b: \mathcal{U}_\chi(\mathfrak{g}) \times \mathcal{U}_\chi(\mathfrak{g}) \to k.\]
In \cite{berkson}, the case $\chi=0$ is treated. In \cite{friedlander1988modular} it is generalized to arbitrary $\chi$. We will sketch the construction of this bilinear form in the following lemma.

Let $x_i$ be a basis for $\mathfrak{g}$, so $\prod_i x_i^{e_i}$ for $e_i\in \{0,1,\ldots, p-1\}$ form a basis for $\mathcal{U}_0(\mathfrak{g})$. We define $S$ as the span of $\prod_i x_i^{e_i}$ where the $e_i$ are not all $p-1$. The bilinear form $b$ constructed in \cite{berkson} is defined as $b(u,v)=\varphi_0(uv)$, where the linear map $\varphi_0$ is given by $\varphi_0|_S = 0$ and $\varphi_0(\prod_i x_i^{p-1})=1$. The same definition works in the general case.

\begin{lemma-N}[\cite{friedlander1988modular}]
There exists a non-degenerate associative bilinear form
\[b: \mathcal{U}_\chi(\mathfrak{g}) \times \mathcal{U}_\chi(\mathfrak{g}) \to k.\]
\end{lemma-N}

\begin{proof}
We claim that there is a $Z_{\mathrm{Fr}}$-bilinear pairing
\[\mathcal{U}(\mathfrak{g}) \otimes_{Z_{\mathrm{Fr}}} \mathcal{U}(\mathfrak{g}) \to Z_{\mathrm{Fr}},\]
defined via the pairing $b$ on $\mathcal{U}_0(\mathfrak{g})$. The algebra $\mathcal{U}(\mathfrak{g})$ is known to be a free module over $Z_{\mathrm{Fr}}$ as a consequence of the PBW theorem (see \cite{jantzen1998representations}). The elements $\prod_i x_i^{e_i}$ for $e_i\in \{0,1,\ldots, p-1\}$ form a basis for $\mathcal{U}(\mathfrak{g})$ as a $Z_{\mathrm{Fr}}$-algebra.  Define a lifted map $\Phi_0$ in the same way as $\varphi_0$, as the $Z_{\mathrm{Fr}}$-linear map given by
\[\Phi_0(\prod_i x_i^{e_i}) = \begin{cases} 1 \text{ if } (e_1, \ldots, e_n)=(p-1,\ldots, p-1) \\ 0 \text{ otherwise.} \end{cases}\]
We define the pairing as $\widetilde{b}(x,y)=\Phi_0(xy)$. The same argument as in \cite{berkson} is applicable and shows non-degeneracy and associativity.

Let $k_\chi$ be the one dimensional $Z_{\mathrm{Fr}}$-algebra determined by the character (as a map $\mathfrak{g} \to k$). Noting that $\mathcal{U}_\chi(\mathfrak{g}) = \mathcal{U}(\mathfrak{g}) \otimes_{Z_{\mathrm{Fr}}} k_\chi$, upon a base change $Z_{\mathrm{Fr}} \to k_\chi$ we see $\widetilde{b}$ induces the desired pairing on any $\mathcal{U}_\chi(\mathfrak{g})$.
\end{proof}

If we let $N=\dim_k \mathfrak{g}$ and $V_i$ be as in \S \ref{sect: filtration_def} we have a filtration of length $N(p-1)+1$ on $\mathcal{U}_\chi(\mathfrak{g})$
\[ 0 \subset V_0 \subset V_1 \subset \ldots \subset V_{N(p-1)} = \mathcal{U}_\chi(\mathfrak{g}).\]
Here, $0$ denotes the trivial vector space. From the competibility of the filtration with multiplication and the definition of the bilinear form $b$ we see that \[V_i^{\perp} = V_{N(p-1)-i-1},\] where $\perp$ denotes the orthogonal complement with respect to $b$.

Additionally, suppose we have a block decomposition $\mathcal{U}_\chi(\mathfrak{g})=\bigoplus_\lambda A_\lambda$. The $A_\lambda$ are mutually orthogonal with respect to the algebra structure of $\mathcal{U}_\chi(\mathfrak{g})$, and hence we deduce that the bilinear form $b$ restricts to $b_\lambda: A_\lambda \times A_\lambda \to k$ which is still a non-degenerate associative pairing. We denote the orthogonal complement in $A_\lambda$ with respect to $b_\lambda$ by $\perp_\lambda$.

\begin{theorem-N}
Let $V=V_i$ for some $i$. Then $V^\perp$ is also a term of the PBW-filtration, and
\[\pi_\lambda(V)^{\perp_\lambda} = A_\lambda \cap V^\perp.\]
In particular, we have $\dim \pi_\lambda(V) + \dim A_\lambda \cap V^\perp = \dim A_\lambda$.
\label{duality_theorem}
\end{theorem-N}

\begin{proof}
The assertion $\pi_\lambda(V)^{\perp_\lambda} = A_\lambda \cap V^{\perp}$ is equivalent to $\pi_\lambda(V) = (A_\lambda \cap V^\perp)^{\perp_\lambda}$. We have
\[(A_\lambda \cap V^\perp)^{\perp_\lambda} = (A_\lambda \cap V^\perp)^{\perp} \cap A_\lambda = (A_\lambda^{\perp} + V) \cap A_\lambda.\]
Now we recall that the map $\pi_\lambda$ is a projection map as it is a central idempotent, and also that the kernel is precisely $A_\lambda^\perp$. This is due to the definition of $b$, since if $uv=0$ we have $b(u,v)=0$. Thus, we have $\ker \pi_\lambda = \bigoplus_{\lambda'\neq \lambda} A_{\lambda'} = A_{\lambda}^{\perp}$.

Since $\pi_\lambda|_{A_\lambda}=\id$ we have
\[\pi_\lambda((A_\lambda^{\perp} + V) \cap A_\lambda) = (A_\lambda^{\perp} + V) \cap A_\lambda.\]
Knowing the kernel of $\pi_\lambda$, the left side is $\pi_\lambda(V)$. Thus $\pi_\lambda(V) = (A_\lambda \cap V^\perp)^{\perp_\lambda}$.
\end{proof}

Due to this result, it suffices to study just one of the filtrations of Definition \ref{pbw_on_blocks_def}. In what follows we understand the pushforward filtration using its relation to the shifted filtration.

\section{The case of \texorpdfstring{$\mathfrak{g}=\sl$}{sl2}}

\noindent \textbf{\textsc{Notation.}} Throughout this section we will refer to the block by the action of the central element $c$ (defined by Lemma \ref{HCcenter}) on the simple modules in the block, which is a constant we denote by $\alpha$ (similar to the $\alpha_i$ of \S \ref{sect: shifted_PBW}). We also assume $p>2$ throughout.

\subsection{Central idempotents} We start by providing explicit formulas for the central idempotents corresponding to the block and defining the projections $\pi_\alpha: \mathcal{U}_\chi(\mathfrak{g}) \to A_\alpha$.  In order to do this, we compute the image of $Z_{\mathrm{HC}}$ in $\mathcal{U}_\chi(\sl)$. We will denote this image by $Z$.

\begin{proposition-N}
We have 
\begin{enumerate}
    \item When $\chi=0$ or $\chi=e$ the image of the center is
    \[Z \simeq k[c]/\langle c^p - 2c^{(p+1)/2}+c \rangle.\]
    \item When $\chi = a h/2$ the image of the center is
    \[Z \simeq k[c]/\langle c^p - 2c^{(p+1)/2}+c - a^2 \rangle.\]
\end{enumerate}
\label{center_relations}
\end{proposition-N}

\begin{proof}
By results of \ref{sect: center} we have 
\[Z = k_\chi \otimes_{Z_{\mathrm{Fr}}^G} Z_{\mathrm{HC}}.\]
We identify $Z_{\mathrm{HC}}$ with $k[(h-1)^2] = k[h'^2]$ via Lemma \ref{HCcenter}. As this is a PID we need only find the relation induced by the tensor product. We may identify
\[Z_{\mathrm{Fr}}^G \simeq \mathrm{S}(\mathfrak{h}^{(1)})^{W,\bigcdot} = k[h'^p-h']^{W,\bigcdot} = k[(h'^p-h')^2].\]
Knowing this, we see that once we expand we get the relation $h'^{2p}-2h'^{p+1}+h'^2=0$ in case (1) of the proposition. Using our preimage $c$ of the generator $h'^2$, once we reverse the Harish-Chandra isomorphism we get $k[c]/\langle c^p - 2c^{(p+1)/2}+c \rangle$ for the center.

In the same way in case (2) we get the relation  $(h'^p-h')^2-a^2=0$ for $h'$ since $h'^p-h'$ acts by $a$ on $k$, and hence \[Z \simeq k[c]/\langle c^p - 2c^{(p+1)/2}+c - a^2 \rangle.\]
\end{proof}

Now we are ready to compute the central idempotents $\pi_\alpha$ for the block decompositions of $\mathcal{U}_\chi(\sl)$.

\begin{theorem-N}
As in Proposition \ref{center_relations}, we have two cases for the central idempotents of the block decomposition:
\begin{enumerate}
    \item Set $\Phi(c) = c^p-2c^{(p+1)/2}+c$. When $\chi=0$ or $e$, the block decomposition becomes
    \[\mathcal{U}_\chi(\sl) = \bigoplus_{\qr{\alpha}{p}\neq -1} \pi_\alpha \mathcal{U}_\chi(\sl)\]
    where $\pi_0 = \Phi(c)/c$ and $\pi_\alpha = \frac{2(c+\alpha)\Phi(c)}{(c-\alpha)^2}$ for $\alpha \neq 0$ a quadratic residue modulo $p$.
    \item Set $\Phi_a(c) = c^p-2c^{(p+1)/2}+c - a^2$ for $\alpha \in k^\times$. We have for $\chi = a h/2$ the block decomposition
    \[\mathcal{U}_\chi(\sl) = \bigoplus_{\Phi_a(\alpha)=0} \pi_\alpha \mathcal{U}_\chi(\sl),\]
    where $\pi_\alpha = \kappa_\alpha \frac{\Phi_a(c)}{c-\alpha}$, where $\kappa_\alpha$ is a normalizing constant so that $\pi_\alpha(\alpha)=1$.
\end{enumerate}
\label{block_idempotents}
Here, we abuse notation by writing the idempotents $\pi_\alpha$ as elements of $k[c]$, and then passing to $k[c]/\langle \Phi(c)\rangle$ or $k[c]/\langle \Phi_a(c) \rangle$ to obtain the actual idempotents.
\end{theorem-N}

We will break the proof of this theorem into several parts.

\begin{lemma-N}
Dividing by cases in Theorem \ref{block_idempotents}, we have:
\begin{enumerate}
    \item When $\chi=0$ or $e$, if $\alpha'\in \F_p$ is a square and we view $\pi_\alpha$ as an element of $k[c]$ then $\pi_\alpha(\alpha') = 0$ unless $\alpha'=\alpha$, in which case we get $1$.
    \item When $\chi = ah/2$ for $a \in k^\times$ and $\alpha'$ is a root of $\Phi_a$, we have $\pi_\alpha(\alpha')=0$ unless $\alpha'=\alpha$, in which case we get $1$.
\end{enumerate}
\end{lemma-N}

\begin{proof}
We begin with case (1). A helpful observation in this case is that $c^p-2c^{(p+1)/2}+c$ already factors completely over $\F_p \subset k$ as $c\prod_{\qr{\alpha}{p}=1}(c-\alpha)^2$. Then we have
\[\pi_0 = \prod_{\qr{\alpha}{p}=1}(c-\alpha)^2,\]
and so for any nonzero square we get $\pi_0(\alpha')=0$. At $0$ we get a product of squares of quadratic residues. This is $1$, because $\prod_{k=1}^{\frac{p-1}{2}} k^2 \equiv (-1)^{\frac{p-1}{2}} (p-1)! \equiv (-1)^{\frac{p-1}{2}} \pmod{p}$ by Wilson's theorem. For $\alpha \neq 0$, we can use the factorization to obtain
\[\pi_\alpha = 2c (c+\alpha) \prod_{\alpha' \neq \alpha, \qr{\alpha'}{p}=1} (c-\alpha')^2.\]
Obviously, this implies the value is $0$ at any square not equal to $\alpha$. Note that as a polynomial in $k[c]$ we have
\[\pi_\alpha(\alpha) =  4 \alpha^2 \prod_{\alpha \neq \alpha', \qr{\alpha'}{p}=1} (\alpha-\alpha')^2 = 4\alpha^2 \alpha^{p-3} \prod_{\alpha' \neq 1, \qr{\alpha'}{p}=1} (1-\alpha')^2 \equiv 4 \prod_{\alpha' \neq 1, \qr{\alpha'}{p}=1} (1-\alpha')^2.\]
This is equal to $4\frac{\Phi(c)}{c(c-1)^2}\big|_1$. This can be computed explicitly over $\Z$ and the value at $1$ is $4\left(\frac{p-1}{2}\right)^2$, and so reducing to $\F_p[c] \subset k[c]$ we see that the value is equal to $1$.

In case (2), by definition $\Phi_a(c) = \prod_{\Phi_a(\alpha)} (c-\alpha)$ and all roots have multiplicity $1$. The statement now immediately follows from the definition of $\kappa_\alpha$ and $\pi_\alpha$.
\end{proof}

\begin{lemma-N}
The elements $\pi_\alpha$ in all cases are central idempotents, and $\sum_\alpha \pi_\alpha = 1$.
\end{lemma-N}

\begin{proof} By construction $\pi_\alpha$ are central. Consider first the case $\chi=0$ or $e$.

Let us check $\pi_\alpha^2=\pi_\alpha$. For $\alpha = 0$, we observe that
\[\pi_0^2 = (\Phi(c)-\alpha \pi_0) \frac{1}{c-\alpha}\prod_{\qr{\alpha}{p}=1}(c-\alpha)^2.\]
Reducing modulo $\Phi(c)$ and repeatedly applying this relation we get $\left(\prod_{\qr{\alpha}{p}=1} \alpha^2\right) \pi_0 = \pi_0 \pmod{\Phi(c)}$. Thus, $\pi_0^2=\pi_0$. Using a similar reduction method, we can show $\pi_\alpha$ is idempotent.

Now we show $\sum_\alpha \pi_\alpha = 1$. We claim that this relation holds already in $k[c]$. By the previous lemma, as a polynomial in $k[c]$ the element $\sum_\alpha \pi_\alpha$ on each quadratic residue modulo $p$ has value $1$. For nonresidues $\eta, \eta' \in \F_p$ and any $\alpha$ we have 
\[\pi_\alpha(\eta) = \pi_{(\eta'/\eta)\alpha}(\eta').\] 
To see this, note that in $\F_p$ for a nonresidue $\eta$ we have $\prod_{\qr{\alpha}{p}=1} (\eta - \alpha)^2 = \prod_{\qr{\eta'}{p}=-1}(1-\eta')^2 = 4$. This gives the claim for $\alpha = 0$. The rest comes from the relation $2\eta' (\eta'+\alpha\frac{\eta'}{\eta}) = (\eta/\eta')^2 2\eta(\eta+\alpha)$, since upon pulling this factor into the product $\prod_{\alpha' \neq (\eta'/\eta)\alpha} (\eta' - \alpha')^2$ in the factorization of $\pi_{(\eta'/\eta)\alpha}(\eta')$ we get $\pi_\alpha(\eta)$. We conclude that \[\sum_\alpha \pi_\alpha(\eta)=\sum_\alpha \pi_{(\eta'/\eta)\alpha}(\eta')=\sum_\alpha \pi_\alpha(\eta').\]
Hence, we may deduce that $\sum_\alpha \pi_\alpha \in k[c]$ is a polynomial of degree $<p-1$ which is constant on the set of quadratic residues and on the set of nonresidues modulo $p$. The bound on the degree holds because we know that the leading coefficients of the $\pi_\alpha$ sum to $0$ in $\F_p \subset k$.

Due to its degree, the residue modulo $(c^{(p-1)/2}-1)(c^{(p-1)/2}+1)$ is uniquely determined. By the Chinese remainder theorem, its residues modulo $(c^{(p-1)/2}\pm 1)$ uniquely determine the result. That is, we have
\[\sum_\alpha \pi_\alpha = x(c^{(p-1)/2}-1)+y(c^{(p-1)/2}+1).\]
Now evaluating on a quadratic residue gives $2y = 1$, so $y \equiv \frac{1}{2} \pmod{p}$. We can evaluate at $c=0$ to get the constant term as $1$, so we force $x=-y$. This implies $\sum_\alpha \pi_\alpha=1$ as desired.

Since in case (2) the polynomial $\Phi_a(c)$ has simple roots the claim follows from the previous lemma and the Chinese remainder theorem.
\end{proof}

We may now prove the theorem.
\begin{proof}[Proof of Theorem \ref{block_idempotents}]
The previous lemma shows that in any case the $\pi_\alpha$ are central idempotents. It is easy to check that they are mutually orthogonal (via Proposition \ref{center_relations}), and since they sum to $1$ they define a direct sum decomposition 
\[\mathcal{U}_\chi(\sl) = \bigoplus_\alpha \pi_\alpha \mathcal{U}_\chi(\sl).\]

What remains to check is that this is indeed the block decomposition. For a block corresponding to $\alpha$, by definition the action of $c$ on the simples in the block is equal to $\alpha$. It follows that for a constant $\beta\neq\alpha$ the action of $c-\beta$ is invertible on the block, by the definition of the Serre subcategory. We claim that $(c-\alpha)^{d_\alpha}$ acts by $0$ on this block, where $d_\alpha$ is the degree of the factor $c-\alpha$ in $\Phi(c)$. Indeed, the action of $\Phi(c)$ on the object in the block is equal up to the invertible element to the multiplication by $(c-\alpha)^{d_\alpha}$ but on the other hand it is $0$ by Proposition \ref{center_relations}. Alternatively, this may be seen from the construction of Theorem \ref{liealgcharblocks}, because the action of $(c-\alpha)^{d_\alpha}$ is $0$ on the indecomposable projective covers for the block (see Lemma \ref{projective_covers}).

It follows that $\pi_\alpha$ vanishes on the blocks other then the one corresponding to $\alpha$. We conclude that $\pi_\alpha\mathcal{U}_\chi(\sl)\subset A_\alpha$. The statement follows by a dimension argument.
\end{proof}

An immediate corollary of Theorem \ref{block_idempotents} is an understanding of the algebra of coinvariants for every block, as we now know the maps $\pi_\alpha$ and the image of the Harish-Chandra center explicitly.

\begin{corollary-N}
We have two results corresponding to the cases of the previous theorem.
\begin{enumerate}
    \item For $\chi = 0$ or $e$, the algebra of coinvariants for the block $A_\alpha$ (indexed by squares in $\F_p$) is $\frac{k[c]}{(c-\alpha)^2}$ for $\qr{\alpha}{p}=1$ and $\frac{k[c]}{c}$ for $\alpha=0$.
    \item For $\chi = ah/2$ regular, the algebra of coinvariants for $A_\alpha$ (indexed by roots of $\Phi_a$) is $\frac{k[c]}{c-\alpha}$.
\end{enumerate}
\label{coinvariant_algebras}
\end{corollary-N}
\begin{proof}
The statement follows immediately from the fact that $C_\alpha\simeq \pi_\alpha Z$.
\end{proof}

This is consistent with Lemma \ref{harishchandrakernel}, since these are all local rings and graded (all characters of $\sl$ are either regular or nilpotent). We also conclude that for $\chi$ regular $A_\alpha = \ker (c-\alpha)$, and for $\chi=0$ or $e$ it is $A_0 = \ker c$ and $A_\alpha = \ker (c-\alpha)^2$ for $\alpha\ne 0$.

\subsection{Determining the associated graded algebras} 

Throughout this and next sections we will assume $\chi=0$ unless the opposite is specified explicitly.

Recall that under this assumption Theorem \ref{subschemestruct} provide an $\SL_2$-equivariant map
\begin{display}
k[\mathcal{N}_p] \otimes_k C_\alpha  \ar{r} & \gr_{\mathrm{sh}} A_\alpha.
\end{display}

We determined $C_\alpha$ in the previous section. Let us now describe $k[\mathcal{N}_p]$ as a representation of $\SL_2$ and a graded algebra.

Consider a matrix $A=\begin{pmatrix} a & b \\ c & -a \end{pmatrix}$ in $\sl$. It lies in the nilpotent cone if and only if $A^2=0$, which happens if and only if $a^2+bc=0$. Thus, we have
\[\mathcal{N} = \Spec k[a,b,c]/\langle a^2+bc\rangle.\]
Recall that the Frobenius neighborhood of the nilpotent cone, $\mathcal{N}_p$, is defined as the preimage of $0$ under the map $\mathrm{Fr}: \mathcal{N} \to \mathcal{N}^{(1)}$. Thus, for $\sl$ we can write
\[k[\mathcal{N}_p] = k[a,b,c]/\langle a^p,b^p,c^p, a^2+bc \rangle.\]
There is an $\SL_2$-action on $\mathcal{N}$ via the adjoint action on $\mathfrak{g}$. This induces an action on $\mathcal{N}_p$, and the structure sheaf becomes an equivariant sheaf for this action so that $k[\mathcal{N}_p]$ itself is an $\SL_2$-module. The action is compatible with the degree, so we may decompose the degree $d$ component of $k[\mathcal{N}_p]$ as an $\SL_2$-module. 

\begin{proposition-N}
We have
\[k[\mathcal{N}_p]_d = \nabla_{2d}\]
for $d<p$, and for $p$ through $3(p-1)/2$ (inclusive) it is $L_{4p-2d-2}$ or $0$ after this point. That is, the dimensions of graded components in the associated graded algebra are
\[1, 3, \ldots, 2p-1, 2p-2, 2p-6, \ldots, 4\]
so that $\dim k[\mathcal{N}_p] = p^2 + \frac{p^2-1}{2}$.
\label{nilcone_sl2_decomp}
\end{proposition-N}

\begin{proof}
The ideal $\langle a^p, b^p, c^p \rangle$ is concentrated in degree $\ge p$. Hence, for $d<p$ the graded components of $k[\mathcal{N}_p]$ are the same as in $k[\mathcal{N}]$. We claim these are $\nabla_{2d}$. 

Consider the isomorphism
\begin{display}
\Spec k[x,y]/\{\pm 1\} \ar{r}{\simeq} & \mathcal{N}
\end{display}
where $\pm 1$ act by scalars, given by $\pm (x,y)\mapsto\begin{pmatrix} xy & -y^2 \\ x^2 & -xy \end{pmatrix}$. It is compatible with the $\SL_2$-action and preserves grading up to rescaling by $2$. The $\pm$-invariants on $k[x,y]$ are exactly even graded components, which are $\nabla_{2d}$ as $\SL_2$-representations.

In degree $p$ the elements $a^p, b^p, c^p$ generating the kernel ideal of the map $k[\mathcal{N}]\to k[\mathcal{N}_p]$ form a $G$-submodule $L_{2p}\simeq \mathrm{Fr}^* L_2$. For $p\le d \le 3(p-1)/2$, we have an exact sequence
\begin{display}
0 \ar{r} & L_{2d} \ar{r} & \nabla_{2d} \ar{r} & L_{4p-2d-2} \ar{r} & 0 \\
 &  \mathrm{Fr}^* L_2 \otimes L_{2d-2p} \ar{u}{\simeq} & & 
\end{display}
The isomorphism $L_{2d} \simeq \mathrm{Fr}^* L_2 \otimes L_{2d-2p}$ follows via Lemma \ref{frob_simple_twist} and this simple submodule is precisely the image of multiplication of the elements in $k[\mathcal{N}]_{d-p}$ by the generators of the ideal. Hence we conclude
\[k[\mathcal{N}_p]_d = L_{4p-2d-2}\]
for $d$ in this range. At $d=\frac{3(p-1)}{2}+1$, we have an isomorphism
\begin{display}
\mathrm{Fr}^* L_2 \otimes \nabla_{p-1} \ar{r}{\sim} & \nabla_{3p-1}
\end{display}
since the dimension of the left side is $(\dim \mathrm{Fr}^* L_2)(\dim \nabla_{p-1}) = 3p = \dim \nabla_{3p-1}$ and the map is an inclusion. Hence, the graded components are $0$ from this point onward.
\end{proof}
\begin{remark}
In characteristic $0$ the description of the $G$-representation structure on $k[\mathcal{N}]$ for general $G$ is provided in \cite{kostant1963lie}. In particular, for $G=\SL_2$ the description is identical to ours.
\end{remark}

Consider a block subalgebra $A_\alpha$ of $\mathcal{U}_0(\sl)$. We put simply $\langle c-\alpha \rangle$ for the ideal $(c-\alpha) A_\alpha$ To understand the map
\begin{display}
k[\mathcal{N}_p] \otimes_k C_\lambda  \ar{r} & \gr_{\mathrm{sh}} A_\alpha.
\end{display}
we will use a map $k[\mathcal{N}_p] \to \gr_{\mathrm{sh}} A_\alpha/\langle c-\alpha \rangle$ and the multiplication by $(c-\alpha)$ map
\begin{display}
\gr_{\mathrm{sh}} A_\alpha/\langle c-\alpha \rangle  \ar{r}{\cdot (c-\alpha)} & \gr_{\mathrm{sh}} \langle c-\alpha \rangle,
\end{display}
which is well defined since $(c-\alpha)^2=0$ and shifts degree by $1$. We note here that the grading on the ideal $\langle c-\alpha \rangle$ is given by intersecting the ideal with the filtration on $A_\alpha$.

\begin{lemma-N}
We have a surjective map $\SL_2$-equivariant map
\begin{display}
k[\mathcal{N}_p] \ar[two heads]{r} & \gr_{\mathrm{sh}} A_\alpha/\langle c-\alpha \rangle.
\end{display}
When $\chi\neq 0$, this is still well-defined and surjective but need not be $\SL_2$-equivariant.
\label{surj_quotient}
\end{lemma-N}

\begin{proof}
This is a direct consequence of Theorem \ref{subschemestruct} and its proof. 
\end{proof}

In terms of the generators $a,b,c$ the above map is given explicitly by 
\[a \mapsto \overline{\pi_\alpha(h)}, b \mapsto \overline{\pi_\alpha(2e)}, c \mapsto \overline{\pi_\alpha(2f)},\] where we put the overline to denote the image of the element in $A_\alpha/\langle c-\alpha \rangle$ and then passing to the associated graded algebra. Note that the images of $a,b,c$ generate $\gr_{\mathrm{sh}} A_\alpha/\langle c-\alpha \rangle$, so this map is surjective.

In the case of $\alpha=0$ for the trivial character, we know $\pi_\alpha(c-\alpha) = 0$ by Corollary \ref{coinvariant_algebras}. Thus, $\langle c-\alpha \rangle=0$. The following theorem shows in the rest of the cases that this and the multiplication by $c-\alpha$ maps together suffice to understand the map of Theorem \ref{subschemestruct}.

\begin{theorem-N}
Fix a block $A_\alpha$ for $\chi=0, e,$ or $ah/2$ as in Lemma \ref{charactercases}. with algebra of coinvariants $k[c]/\langle (c-\alpha)^2 \rangle$. Consider the surjective map
\begin{display}
k[\mathcal{N}_p] \otimes_k \frac{k[c]}{(c-\alpha)^2}  \ar[two heads]{r} & \gr_{\mathrm{sh}} A_\alpha
\end{display}
of Theorem \ref{subschemestruct}. A pure tensor $x\otimes 1$ for $x$ homogeneous is in the kernel when $x$ is in the kernel of $k[\mathcal{N}_p] \to \gr_{\mathrm{sh}} A_\alpha/\langle c-\alpha \rangle$. A pure tensor $x \otimes (c-\alpha)$ for $x$ homogeneous is in the kernel if $x$ is in the kernel of the composite map
\begin{display}
k[\mathcal{N}_p] \ar{r} & \gr_{\mathrm{sh}} A_\alpha/\langle c-\alpha \rangle \ar{r}{\cdot (c-\alpha)} & \gr_{\mathrm{sh}} \langle c-\alpha \rangle.
\end{display}
This describes the kernel restricted to $k[\mathcal{N}_p]\otimes 1$ and  $k[\mathcal{N}_p] \otimes (c-\alpha)$. The kernel is a direct sum of these kernels as vector spaces.
\label{sl2_subschemestruct}
\end{theorem-N}

\begin{proof}
From the construction of the map and the proof of Theorem \ref{subschemestruct} it follows that the image of $k[\mathcal{N}_p]$ does not intersect the image of the ideal $\langle c-\alpha \rangle$ in the associated graded. This implies the assertion about the elements $x\otimes1$.

By construction of the map, the element $(c-\alpha)$ in the coinvariants algebra $C_\alpha$ maps to the image of $\pi_\alpha(c-\alpha)$ in the associated graded. Then $x\otimes (c-\alpha)$ maps to the image of $x$ under
\begin{display}
k[\mathcal{N}_p] \ar{r} & \gr_{\mathrm{sh}} A_\alpha/\langle c-\alpha \rangle \ar{r}{\cdot (c-\alpha)} & \gr_{\mathrm{sh}} \langle c-\alpha \rangle.
\end{display}
and assertion about such elements follows.

The restrictions of the kernel to $k[\mathcal{N}_p]\otimes 1$ and  $k[\mathcal{N}_p] \otimes (c-\alpha)$ do not intersect and generate the kernel. This yields the splitting.
\end{proof}

We apply Theorem \ref{sl2_subschemestruct} in the case of $\chi=0$ and $\alpha \neq 0$, because here we can exploit the $\SL_2$-invariance of the kernel to compute it. With this theorem, we have reduced the problem to understanding the multiplication by $(c-\alpha)$ maps and the map $k[\mathcal{N}_p] \to \gr_{\mathrm{sh}} A_\alpha/\langle c-\alpha \rangle$. We will understand $\gr_{\mathrm{sh}} \langle c-\alpha \rangle$ first.

\begin{lemma-N}
Let $\alpha \neq 0$. In $A_\alpha$, we have
\[e^{p-\omega}(c-\alpha) = 0\]
where we pick $\omega$ such that $\omega^2 = \alpha$ and $1 \le \omega \le \frac{p-1}{2}$.
\label{nontrivial_kernel}
\end{lemma-N}

\begin{proof}
We may decompose the block subalgebra $A_\alpha$ into indecomposable projectives over itself as $\bigoplus_i \mathcal{P}_i$, since it is Artinian (it has finite dimension over $k$). Each of these will be a projective cover of a simple module. On each $\mathcal{P}_i$, the image of multiplication by $(c-\alpha)$ will be an irreducible module so that $(c-\alpha)\mathcal{P}_i = L$ or $0$, where $L$ is a simple module for $A_\alpha$. This follows from Lemma \ref{projective_covers} and that $c-\alpha$ annihilates the baby Verma modules in the corresponding block.

There are two isomorphism classes of projective covers for simple modules of $A_\alpha$, which are related to each other by the action of the Weyl group $W=\Z/2\Z$ (\cite{cline1985injective}, appendix). This action preserves $c-\alpha$, so if $c-\alpha=0$ on one isomorphism class it is $0$ on the other. As it does not act by $0$ on the block and has a nontrivial kernel, we conclude $(c-\alpha)\mathcal{P}_i=L$. We know there are two such simple modules up to isomorphism via the block decomposition, namely $L_{0,p-\omega-1}$ and $L_{0,\omega-1}$. Then in either case $e^{p-\omega} L = 0$, from which the identity follows.
\end{proof}

\begin{lemma-N}
Resume the notation and assumptions of the previous lemma. We have $\dim \langle c-\alpha\rangle = \omega^2+(p-\omega)^2$ and an isomorphism
\begin{display}
A_\alpha/\Ann(c-\alpha) \ar{r}{\sim} & \End_k(L_{0,\omega-1}) \oplus \End_k(L_{0,p-\omega-1}).
\end{display}
\label{quotient_dimensions}
\end{lemma-N}

\begin{proof}
We start by constructing a map
\begin{display}
A_\alpha/\Ann(c-\alpha) \ar{r} & \End_k(L_{0,\omega-1}) \oplus \End_k(L_{0,p-\omega-1}).
\end{display}
We may again make a direct sum decomposition into indecomposable projectives $A_\alpha = \bigoplus_i \mathcal{P}_i$. This shows $(c-\alpha)A_\alpha = \bigoplus_i L_i$ where $L_i$ are simple $\sl$-modules which are either $L_{0,\omega-1}$ or $L_{0,p-\omega-1}$. The multiplicity of each equals their dimension by the results stated in \cite{friedlander1988modular}, and hence $\dim \langle c-\alpha\rangle = \omega^2+(p-\omega)^2$.

These are both simple modules for $A_\alpha/\Ann(c-\alpha) \simeq \langle c-\alpha \rangle$ due to being simple $\sl$-modules - we may explicitly define the action of an element of $A_\alpha/\Ann(c-\alpha)$ through the corresponding action under the multiplication by $(c-\alpha)$ map on $(c-\alpha)\mathcal{P}_i$. These are all the simple modules for $A_\alpha/\Ann(c-\alpha)$ because they are the only simple modules for $A_\alpha$.

Now we show the map $A_\alpha/\Ann(c-\alpha) \to \End(L_{0,\omega-1}) \oplus \End(L_{0,p-\omega-1})$ is an isomorphism. The simple $A_\alpha/\Ann(c-\alpha)$-modules  $L_{0,\omega-1}, L_{0,p-\omega-1}$ are projective because they are direct summands of $A_\alpha/\Ann(c-\alpha)$. Then all of the simple modules are projective, which implies that $A_\alpha/\Ann(c-\alpha)$ is a semisimple algebra. The claim follows from the Artin-Wedderburn theorem.
\end{proof}

\begin{remark}
When $\chi=e$, we have $A_\alpha \simeq \Mat_p(C_\alpha)$ where $C_\alpha = k[c]/(c-\alpha)^2$ is the algebra of coinvariants from Corollary \ref{coinvariant_algebras} by \cite{brownblocks} \S 3.19. Then $\dim_k A_\alpha = 2p^2$, and $\dim_k \langle c-\alpha \rangle = p^2$. We may also use the same argument as above to show $A_\alpha/\Ann(c-\alpha) \simeq \End_k(\Delta_{e,\alpha})$ - the only difference is that $\Delta_{e,\alpha}$ is the only simple that appears.
\end{remark}

These observations will allow us to obtain the following result.

\begin{proposition-N}
For the trivial character, via the map of Lemma \ref{surj_quotient} we have an isomorphism
\[\gr_{\mathrm{sh}} A_\alpha/\langle c-\alpha \rangle \simeq k[\mathcal{N}_p]_{< p+\omega}.\]
Given $\alpha$ we choose $\omega$ so $\alpha = \omega^2$ for $0\le \omega \le \frac{p-1}{2}$.
\label{restricted_degree}
\end{proposition-N}

\begin{proof}
In degrees $\ge p$ we identify all graded components with simple $\SL_2$-modules in $k[\mathcal{N}_p]$ in Proposition \ref{nilcone_sl2_decomp}. Because the kernel can be checked to be $\SL_2$-invariant by Lemma \ref{surj_quotient}, if there is a nontrivial kernel in some degree $d\ge p$ then the kernel includes all degrees $\ge d$.

Now we show the map of Lemma \ref{surj_quotient} is an isomorphism in degrees below $p$. If there was a kernel in degree below $p$, this would yield a nontrivial kernel in degree $p-1$ so that we get the simple quotient $L_0$ of $\nabla_{2p-2}$ in the image. 

Note that $\gr_{\mathrm{sh}} A_\alpha/\langle c-\alpha \rangle$ is generated by its first graded component $L_2$ and, thus the multiplication map from $L_2\otimes L_0$ is surjective on the $p$-th graded component. On the other hand the $p$-th graded component is a quotient of $k[\mathcal{N}_p]_p=L_{2p-2}$, which is an irreducible $\SL_2$-representation and consequently vanishes. 

We deduce that a nontrivial kernel in degree $p-1$ yields $\gr_{\mathrm{sh}} A_\alpha/\langle c-\alpha \rangle$ vanishing in degrees $\ge p$. By dimension count we see that   $\dim A_\alpha/\langle c-\alpha \rangle <p^2$. This contradicts our earlier dimension findings in Lemma \ref{quotient_dimensions}, so in degrees below $p$ the map of Lemma \ref{surj_quotient} is an isomorphism. Thus, we deduce that $\mathrm{gr} A_\alpha/\langle c-\alpha \rangle$ is $k[\mathcal{N}_p]_{< d_{\alpha}}$ for some $d_\alpha\ge p$. Determining this $d_\alpha$ amounts to computing the dimensions of both objects as vector spaces over $k$.  

We first consider the case of $\alpha\neq 0$. The block $A_\alpha$ has dimension $2p^2$. Using identification $A_\alpha/\Ann(c-\alpha) \simeq (c-\alpha)A_\alpha$ and Lemma \ref{quotient_dimensions}, we conclude that the dimension of $\gr_{\mathrm{sh}} A_\alpha/\langle c-\alpha \rangle$ is
\begin{align*}
2p^2 - \dim_k \End_k(L_{0,\omega-1})\oplus \End_k(L_{0,p-\omega-1}) &= 2p^2 - \left(\omega\right)^2 - \left(p-\omega\right)^2 \\
&= (p+\omega)^2-3(\omega)^2.
\end{align*}
This expands as $p^2+2\omega(p-\omega)$. At $\omega = \frac{p-1}{2}$ this is maximized, and matches the dimension of $k[\mathcal{N}_p]$. Picking $1\le \omega \le \frac{p-1}{2}$ to correspond to each block, we see that having $\omega$ for our block $A_\alpha$ means $d_\alpha = p-1+\omega$.

In the case of $\alpha=0$, the block $A_\alpha$ corresponds to $\Rep(\sl)_{0,0}$ generated by a single simple projective module $L$ and $A_0 \simeq \End_k(L)$ has dimension $p^2$. The result follows from Proposition \ref{nilcone_sl2_decomp} by summing the dimensions up to degree $p-1$ to get $p^2$. This yields the claimed decomposition for $\omega=0$.
\end{proof}

\begin{remark}
For $\chi=ah/2$ and $a\in k^\times$, a block in this case is isomorphic to $\End_k(\Delta_{ah/2,\lambda})$ so it has dimension $p^2$. Additionally, by Corollary \ref{coinvariant_algebras} we see that $c-\alpha=0$ in $A_\alpha$, so this tells us the algebra structure of a block as a quotient of $k[\mathcal{N}_p]$ if we understand the kernel of Lemma \ref{surj_quotient}. In this case, one may explicitly argue there is no kernel in degrees below $p$.
It is sufficient to check this for the map $\mathcal{U}_\chi(\mathfrak{g})/\langle c-\alpha\rangle\to A_\alpha$ induced by $\pi_\alpha$. From the block decomposition we know precisely which elements are annihilated by $\pi_\alpha$, so we can conclude the injectivity in degrees below $p$. We then know that $k[\mathcal{N}_p]/\langle k[\mathcal{N}_p]_p \rangle$ injects into $\gr_{\mathrm{sh}} A_\alpha = \gr_{\mathrm{pf}} A_\alpha$, and as it also has dimension $p^2$ this is an isomorphism of algebras.
\end{remark}

Next, we use Lemma \ref{nontrivial_kernel} to deduce the structure of $\langle c-\alpha\rangle$ for $\chi=0$. When $\alpha=0$, we note this is trivial by Corollary \ref{coinvariant_algebras}.

\begin{proposition-N}
Let $\alpha \neq 0$. We have
\[(\gr_{\mathrm{sh}} \langle c-\alpha\rangle)_d = \begin{cases} \nabla_{2(d-1)} \text{ for } 1\le d \le p-\omega \\ L_{2p-2d} \text{ for } p-\omega < d \le p \\
0 \text{ otherwise} \end{cases}.\]
\label{ideal_decomposition}
\end{proposition-N}

\begin{proof}
We have a surjective map
\begin{display}
\gr_{\mathrm{sh}} A_\alpha/\langle c-\alpha \rangle  \ar{r}{\cdot (c-\alpha)} & \gr_{\mathrm{sh}} \langle c-\alpha \rangle.
\end{display}
This is because the map of algebras $A_\alpha/\langle c-\alpha \rangle \to \langle c-\alpha \rangle$ is surjective and respects the filtrations. Thus, we indeed have a surjective map on the associated graded algebras. It is also a map of graded $\SL_2$-representations, since $(c-\alpha)$ is invariant under the adjoint action of $\SL_2$.

In terms of degrees, since we are in the context of the shifted PBW filtration this sends
\begin{display}
(\gr_{\mathrm{sh}} A_\alpha/\langle c-\alpha \rangle)_d \ar{r} & \mathrm{gr}((c-\alpha)A_\alpha)_{d+1}.
\end{display}
We have a map $k[\mathcal{N}]_{d-1} = \nabla_{2(d-1)} \to \gr_{\mathrm{sh}} ((c-\alpha)A_\alpha)_{d}$ for $1\le d \le p$ as $\nabla_{2(d-1)} = k[\mathcal{N}_p]_{d-1}$ for $d$ in this range. As a result, the image in the degree $d+1$ components for $d$ in this range may be written as quotients of $\nabla_{2(d-1)}$. Using $\SL_2$-invariance, they are then either $\nabla_{2(d-1)}, L_{2p-2(d-1)-2} = L_{2p-2d}$, or $0$.

We showed in Lemma \ref{nontrivial_kernel} that we have $e^{p-\omega}(c-\alpha) = 0$ in $A_\alpha$. We deduce that starting at the degree $p-\omega+1$ the kernel of
\begin{display}
(\gr_{\mathrm{sh}} A_\alpha/\langle c-\alpha \rangle)_d \ar{r} & \gr_{\mathrm{sh}} ((c-\alpha)A_\alpha)_{d+1}
\end{display}
is nontrivial. By the same argument as in the first part of the proof of Proposition \ref{restricted_degree} as the nontrivial kernel in degrees below $p$ yields vanishing in degrees starting from $p$.  

As a result we have a surjection
\begin{display} 
\bigoplus_{0<d\le p-\omega} \nabla_{2(d-1)} \oplus \bigoplus_{p-\omega < d \le p} L_{2p-2d} \ar{r} &  \gr_{\mathrm{sh}} \langle c-\alpha\rangle
\end{display}
of $\SL_2$-modules. By summing up the dimensions of the summands the dimension of the left handside equals to $\omega^2+(p-\omega)^2$. From the Weyl modules, we get $(p-\omega)^2$ and from the simple modules, we get $\omega^2$: starting at $d=p$, we get the odd numbers starting at $1=\dim_k L_{2p-2p}$ for the dimensions. The above lemma shows this is exactly the dimension of $\langle c-\alpha\rangle \simeq A_\alpha/\Ann(c-\alpha)$, so we conclude this is an isomorphism of $\SL_2$-modules and the statement follows.

\end{proof}

These results additionally allow us to write down the $\SL_2$-module structure of $\gr_{\mathrm{sh}} A_\alpha$ when $\chi=0$. Combined with Theorem \ref{sl2_subschemestruct} we may also describe the algebra structure.

\begin{corollary-N}
As an $\SL_2$-module, $\gr_{\mathrm{sh}} A_\alpha$ is a direct sum of $\SL_2$-modules for each graded component where
\[(\gr_{\mathrm{sh}} A_\alpha)_d \simeq (\gr_{\mathrm{sh}} \langle c-\alpha\rangle)_d \oplus (\gr_{\mathrm{sh}} A_\alpha / \langle c-\alpha \rangle)_d.\]
This also determines the dimensions of each graded component. For the algebra structure, we have two cases:
\begin{enumerate}
    \item Let $\alpha=0$. Then $\gr_{\mathrm{sh}} A_\alpha = A_\alpha/\langle c-\alpha \rangle\simeq k[\mathcal{N}_p]_{<p}$.
    \item Take $\alpha\neq 0$. The kernel ideal of the map of  Theorem \ref{sl2_subschemestruct} is generated by elements in degree $p+\omega$ and $p-\omega+1$ respectively corresponding to two options of Theorem \ref{sl2_subschemestruct}. 
\end{enumerate}
\label{sh_SL2_decomp}
\end{corollary-N}

\begin{proof}
We know that we have an exact sequence of graded vector spaces
\begin{display}
0 \ar{r} & \gr_{\mathrm{sh}} \langle c-\alpha\rangle  \ar{r} & \gr_{\mathrm{sh}} A_\alpha \ar{r} & \gr_{\mathrm{sh}} A_\alpha / \langle c-\alpha \rangle \ar{r} & 0.
\end{display}
We claim that it splits as the sequence of graded $\SL_2$-representations. From Proposition \ref{alg_group_blocks} we know that representations from different blocks do not have nontrivial extensions between them. We observe that on each graded component modules in $\gr_{\mathrm{sh}} \langle c-\alpha\rangle$ and $\gr_{\mathrm{sh}} A_\alpha / \langle c-\alpha \rangle$ lie in the different blocks with the exception of the degree $p$ component because of the shift in indexes induced by surjective map $\cdot(c-\alpha)$ between them. In degree $p$, we have $\Ext^1(L_0,L_{2p-2})=0$ (by \cite{cline1979ext1}) so we must have a direct sum in this case.

Now let us consider the claims about the algebra structure. In case (1) we have $\langle c-\alpha=0\rangle$ by Corollary \ref{coinvariant_algebras} and the claim follows from Proposition \ref{restricted_degree}. Now we consider the claims about generators in case (2). The map $k[\mathcal{N}_p] \to \gr_{\mathrm{sh}} \langle c-\alpha\rangle$ is the composite map
\begin{display}
k[\mathcal{N}_p] \ar{r} & \gr_{\mathrm{sh}} A_\alpha/\langle c-\alpha \rangle \ar{r}{\cdot (c-\alpha)} & \gr_{\mathrm{sh}} \langle c-\alpha\rangle.
\end{display}
The first map has kernel generated in degree $p+\omega$ by Proposition \ref{restricted_degree} and the second map has kernel generated in degree $p-\omega+1$ by Proposition \ref{ideal_decomposition}. The rest follows from Proposition \ref{restricted_degree}.
\end{proof}

\begin{remark}
We observe that as a subrepresentation of $\gr_{\mathrm{sh}} A_\alpha$ for $\alpha\neq 0$ the trivial module $L_0$ occurs with the multiplicity $3$. This provide an upper bound on the dimension of the center $Z(A_\alpha)$, which is known to be exact in this case. 

Two of these dimension comes from $C_\alpha$. The third copy of $L_0$ appears as the image of $\pi_\alpha(c-\alpha)$ in degree $p+1$.
\end{remark}

Through the methods similar to the above one may also describe the pushforward filtration.

\begin{proposition-N} We have an exact sequence
\begin{display}
0 \ar{r} & \gr_{\mathrm{pf}} \langle c-\alpha\rangle \ar{r} & \gr_{\mathrm{pf}} A_\alpha \ar{r} & \gr_{\mathrm{pf}} A_\alpha / \langle c-\alpha \rangle \ar{r} & 0, 
\end{display}
where $\gr_{\mathrm{pf}} A_\alpha / \langle c-\alpha \rangle = \gr_{\mathrm{sh}} A_\alpha / \langle c-\alpha \rangle$ and $(\gr_{\mathrm{pf}} \langle c-\alpha\rangle)_{d+1} = (\gr_{\mathrm{sh}} \langle c-\alpha\rangle)_d$. As with the shifted PBW filtration, as an exact sequence of $\SL_2$-modules this splits in each degree.
\label{pf_SL2_decomp}
\end{proposition-N}

\begin{proof}
The existence of such an exact sequence is clear, since we induce this by passing to the associated graded algebras as the maps in the corresponding exact sequence of algebras respect the filtration $V_i^{\mathrm{pf}}$.

The first claim $\gr_{\mathrm{pf}} A_\alpha / \langle c-\alpha \rangle = \gr_{\mathrm{sh}} A_\alpha / \langle c-\alpha \rangle$ holds because these two filtrations agree on the quotient. The second claim about $\gr_{\mathrm{pf}} (c-\alpha)A_\alpha$ follows the exact same argument, except the multiplication by $(c-\alpha)$ map shifts the degree by $2$ instead of $1$ and as a result just shifts components in the associated graded algebra.

By the same method as with $\gr_{\mathrm{sh}} A_\alpha$, the block decomposition tells us that there are no nontrivial extensions of $\SL_2$-modules appearing in the exact sequence of Proposition \ref{pf_SL2_decomp} so we understand the structure as an $\SL_2$-module.
\end{proof}

\begin{remark}
By virtue of Theorem \ref{duality_theorem}, we also understand the dimensions of graded components of $\gr_{\mathrm{int}} A_\alpha$.
\end{remark}

\begin{example}
Let $\mathfrak{g} = \sl$, and suppose $\mathrm{char}(k)=5$. The following table summarizes the dimensions as computed by Proposition \ref{pf_SL2_decomp} for $V_i^{\mathrm{pf}}, V_i^{\mathrm{int}}$ on the blocks $A_0, A_1, A_2$ of dimensions $25, 50$ and $50$.

First, we consider $V_i^{\mathrm{pf}} = \pi_\alpha(V_i)$. We let $*$ denote the trivial vector space.

\begin{center}
\begin{tabular}{c|c|c|c|c|c|c|c|c}
\toprule
$i$                   & 0 & 1 & 2 & 3 & 4 & 5 & 6 & \ldots \\
\hline \hline
$\dim \pi_0(V_i)$     & 1 & 4 & 9 & 16 & 25 & 25 & 25 & \ldots \\

$\dim \pi_1(V_i)$     & 1 & 4 & 10 & 20 & 34 & 49 & 50 & \ldots \\

$\dim \pi_2(V_i)$     & 1 & 4 & 10 & 20 & 34 & 45 & 50 & \ldots \\
\bottomrule
\end{tabular}
\end{center}
\noindent For $\alpha=0$, we get the dimensions of the filtration on $k[\mathcal{N}]$ which are the squares. The other blocks have contributions from both graded algebras in the exact sequence. Next, we put down the dimensions for $V_i^{\mathrm{int}} = V_i \cap A_\alpha$.
\begin{center}
\begin{tabular}{c|c|c|c|c|c|c|c|c|c}
\toprule
$i$                  & $\ldots$ & 5 & 6 & 7 & 8 & 9 & 10 & 11 & 12 \\
\hline \hline
$\dim V_i \cap A_0$    & 0 & 0 & 0 & 0 & 9 & 16 & 21 & 24 & 25 \\

$\dim V_i \cap A_1$    & 0 & 0 & 1 & 16 & 30 & 40 & 46 & 49 & 50 \\

$\dim V_i \cap A_2$    & 0 & 0 & 5 & 16 & 30 & 40 & 46 & 49 & 50 \\
\bottomrule
\end{tabular}
\end{center}
As can be seen, the dimensions in the two tables sum in pairs to the total dimension of the block due to Theorem \ref{duality_theorem}.
\end{example}

Let us discuss the algebraic structure of $\gr_{\mathrm{pf}} A_\alpha$. Due to the projection $\pi_\alpha: \mathcal{U}_\chi(\sl) \to \gr_{\mathrm{pf}} A_\alpha$, we have a map
\begin{display}
\gr \mathcal{U}_\chi(\sl) \simeq \mathrm{S}(\sl)/\langle e^p, f^p, h^p \rangle \ar{r} & \gr_{\mathrm{pf}} A_\alpha.
\end{display}
We wish to determine the kernel of this map using our previous results. By Proposition \ref{restricted_degree}, when $\alpha=0$ we already understand the kernel is generated by the degree $p$ component along with $c-\alpha$ because we may write this as the composition
\begin{display}
\gr \mathcal{U}_\chi(\sl) \simeq \mathrm{S}(\sl)/\langle e^p, f^p, h^p \rangle \ar{r} & k[\mathcal{N}_p] \ar{r} &  \gr_{\mathrm{pf}} A_\alpha.
\end{display}

For $\chi=0$ and $\alpha \neq 0$ the kernel $\mathrm{S}(\sl)/\langle e^p, f^p, h^p \rangle \to \gr_{\mathrm{pf}} A_\alpha$ is generated by $(c-\alpha)^2$ which has degree $4$ and elements in degree $p+\omega$ and $p-\omega+2$. The additional $+1$ comes from $(c-\alpha)$ now having degree two and shifting degrees in Proposition \ref{ideal_decomposition}, as discussed in Proposition \ref{pf_SL2_decomp}. We can see this using the diagram
\begin{display}
                                          & \gr_{\mathrm{pf}} A_\alpha \ar{rd} &                                     \\
\gr \mathcal{U}_\chi(\sl) \ar{ru}{\pi_\alpha} \ar{rr} &                                     & \gr_{\mathrm{pf}} A_\alpha/\langle c-\alpha \rangle
\end{display}
where the bottom map has kernel generated by $c-\alpha$ and elements in degree $p+\omega$. Through our understanding of the all maps besides $\pi_\alpha$, we deduce the stated kernel.

\subsection{The adjoint representation}

Having computed the structure of $\gr_{\mathrm{sh}} A_\alpha$ and $\gr_{\mathrm{pf}} A_\alpha$ as $\SL_2$-modules for $\chi=0$, we can use these results to determine the structure of the adjoint representation for $\mathcal{U}_0(\sl)$. We define it as the result of applying the differentiation functor $D$ to the adjoint action of $\SL_2$ on $\mathcal{U}_0(\sl)$. More explicitly, we extend the adjoint representation of $\sl$ to $\mathcal{U}_0(\sl)$ via the Leibnitz rule, and consider the corresponding $\mathcal{U}_0(\sl)$-representation. 

\begin{proposition-N}
As an $\SL_2$-module we have
\[A_0 \simeq \bigoplus_{i=0}^{(p-1)/2} P_{2i}\]
where $P_{2i}$ is an indecomposable $\SL_2$-module, such that $D(P_{2i})=P_{0,2i}$ for projective modules of Lemma \ref{projective_covers}. %nontrivial extensions of $\nabla_{2i}$ and $\nabla_{2p-2-2i}$ (in the case of $i=(p-1)/2$, it is $M_{2i}=\nabla_{2i})$). After applying the differentiation map $D$, we get $D A_0 \simeq L_{0,p-1}^{\oplus p}$.
\label{regular_projective_SL2}
\end{proposition-N}

\begin{proof}
Recall that we have
\[\gr_{\mathrm{pf}} A_0 \simeq \bigoplus_{d<p} \nabla_{2d}\]
in this case. Note that among the presented Weyl modules only pairs of $\nabla_{2j}$ and $\nabla_{2p-2-2j}$ lie in the same block as $\SL_2$-representations and can form a nontrivial extension. We aim to show that when possible, we do have a nontrivial extension in this decomposition.

Note that $A_0=\End(L_{0,p-1})=L_{0,p-1}\otimes L_{0,p-1}^*=L_{0,p-1}\otimes L_{0,p-1}$ as an $\sl$-representation. In particular, it is a projective $\sl$-module. Since, indecomposable projective $\sl$-modules have dimension $p$ or $2p$, extensions between $D(\nabla_{2j})$ and $D(\nabla_{2p-2-2j})$ are all necessarily nontrivial and, hence, the same is true for $\nabla_{2j}$ and $\nabla_{2p-2-2j}$. We identify the corresponding projective $\sl$-module as being exactly $P_{0,2i}$ by noticing that its socle is equal to the socle of $D(\nabla_{2j})$, which in turn equals to $D(L_{2j})=L_{0,2j}$ for $j\le (p-1)/2$.

%Recall that we have a map $D: \Rep(\SL_2) \to \Rep(\sl)$. With the $\SL_2$-module structure on $\mathcal{U}_0(\sl)$, this map sends it to the adjoint representation on $\mathcal{U}_0(\sl)$. Taking $D A_0 \in \Rep(\sl)$, we can then view this as an $A_0$ module (again recovering the adjoint representation). But $A_0$ is a simple algebra, so the decomposition is $D A_0 \simeq L_{0,p-1}^{\oplus p}$.

%Now we consider the original $\SL_2$-module. If there we no nontrivial extensions between the Weyl modules $\nabla_{2d}$ appearing in the associated graded algebra, these extensions would remain trivial after applying $D$ resulting in a different result. If we do have nontrivial extensions, the only way to do it is as $A_0 \simeq \bigoplus_{i=0}^{(p-1)/2} M_{2i}$ because after this the modules are associated to different blocks and cannot have more nontrivial extensions. Because the direct sum decomposition after applying $D$ is into irreducibles with dimension $p$, this means we must be in the latter case because $p \mid \dim_k M_{2i}$.
\end{proof}

The following lemma allows us to extend the calculation for this block to other blocks for the trivial character.

\begin{lemma-N}
In $A_\alpha/\langle c-\alpha \rangle$, the $\SL_2$-submodule $\overline{V_{p-1}^{\mathrm{pf}}}$ is independent of $\alpha$, where the bar denotes reduction modulo $\langle c-\alpha \rangle$.
\end{lemma-N}

\begin{proof}
We have a map
\begin{display}
\mathcal{U}(\sl)/\langle c-\alpha\rangle  \ar{r} & A_\alpha/\langle c-\alpha \rangle, 
\end{display}
which after passing to the associated graded algebras for the PBW filtration and pushforward PBW filtration on $A_\alpha$ provide a map
\begin{display}
\mathrm{S}(\sl)/\langle c \rangle \ar{r} & \gr_{\mathrm{pf}} A_\alpha/\langle c-\alpha \rangle,
\end{display}
The  quotient $\mathrm{S}(\sl)/\langle c \rangle$ is identified with $k[\mathcal{N}]$ as an algebra and $\SL_2$-module and, thus, we know that the second map is isomorphism in degrees below $p$. It follows that the first map is also an isomorphism in degrees below $p$. 

Note now that the algebra $\mathcal{U}(\sl)$ is free as a module over its Harish-Chandra center $Z_{\mathrm{HC}}\simeq k[c]$. By the definition $\SL_2$ acts trivially on $Z_{\mathrm{HC}}$. It follows that the $\SL_2$-module structure on $\mathcal{U}(\sl)/\langle c-\alpha\rangle$ does not depend on $\alpha$. This implies the claim.
\end{proof}

\begin{corollary}
There is an isomorphism of $\SL_2$-representations $(A_\alpha/\langle c-\alpha \rangle)_{<p}\simeq A_0$. 
\end{corollary}

From this, we may deduce the $\sl$-module structure of both $A_\alpha/\langle c-\alpha \rangle$ and $(c-\alpha)A_\alpha$.

\begin{proposition-N}
For all $\alpha$, we have a decomposition of $\sl$-modules
\[A_\alpha/\langle c-\alpha \rangle = \bigoplus_{0\le i\le(p-1)/2} P_{0,2i} \oplus \bigoplus_{p\le i < p+\omega} L_{0,3p-2i-2}^{\oplus 2}.\] 
For $(c-\alpha)A_\alpha$ when $\alpha \neq 0$, we have a decomposition
\[\bigoplus_{\omega \le i \le \frac{p-1}{2}} P_{0,2i} \oplus \bigoplus_{p-\omega < i \le p} L_{0,2p-2i}^{\oplus 2}.\]

\label{projective_decomp}
\end{proposition-N}

\begin{proof}
The result on $A_\alpha/\langle c-\alpha \rangle$ can be deduced by using Proposition \ref{restricted_degree}. In degrees $<p$ we have a projective $\sl$-submodule, which, hence splits off as a direct summand. Remaining simple $\SL_2$-modules belong to different blocks and, hence, do not have nontrivial extensions between them. It remains to observe that $D(L_{4p-2i-2})=L_{0, 3p-2i-2}^{\oplus 2}$.

The decomposition of $(c-\alpha)A_\alpha$ comes from applying the multiplication by $(c-\alpha)$ map to $(A_\alpha/\langle c-\alpha \rangle)_{<p}\simeq A_0$. By Proposition \ref{ideal_decomposition} we know that in degrees $ \le p-\omega$ it is an isomorphism. Hence for $\omega \le i \le \frac{p-1}{2}$ the modules $P_{2i}$ map isomorphically from $(A_\alpha/\langle c-\alpha \rangle)_{<p}$ onto its image in $(c-\alpha)A_\alpha$. The resulting sum $\bigoplus_{\omega \le i \le \frac{p-1}{2}} P_{0,2i}$ is a projective module and consequently splits off as a direct summand. In degrees $>p-\omega$ the multiplication by $(c-\alpha)$ map is not an isomorphism, and so we are left with a collection of simple modules $L_{0,2p-2i}=D(L_{2p-2i})$ for $p-\omega < i \le p$ with two copies each as possible subquotients. Different such modules belong to different blocks and, therefore, have no nontrivial extensions. The two copies of $L_{0,2p-2i}$ appear as a direct sum because prior to applying $D$ the modules $L_{2p-2i}$ appear as a direct sum as they have no nontrivial extensions by \cite{cline1979ext1}. They lie in different blocks than the other $\SL_2$-modules which appear, and so each remaining simple enters the module as a direct summand.
\end{proof}

\begin{lemma-N}
For $\alpha\ne 0$ as an $\sl$-module, we have
\[A_\alpha = (c-\alpha)A_\alpha \oplus A_\alpha/\langle c-\alpha \rangle.\]
\end{lemma-N}

\begin{proof}
As seen above there are surjective morphisms of $\sl$-representations

\begin{display}
\mathcal{U}(\sl) \ar[two heads]{r}& \mathcal{U}(\sl)/\langle c-\alpha \rangle \ar[two heads]{r} & \bigoplus_{i=0}^{(p-1)/2} P_{0,2i}.
\end{display} 
Let $I$ be the kernel of the composition.

Consider a short exact sequence of $\sl$-representations
\begin{display}
0 \ar{r} & \mathcal{U}(\sl)/I \ar{r}{(c-\alpha)\cdot} & \mathcal{U}(\sl)/(c-\alpha)I \ar{r} & \mathcal{U}(\sl)/\langle c-\alpha \rangle \ar{r} & 0.
\end{display}
As $\mathcal{U}(\sl)/I$ is a projective $\sl$-module the sequence splits. 

Note further that we know that the projection $\mathcal{U}(\sl)\to A_\alpha$ quotients through $\mathcal{U}(\sl)/(c-\alpha)I$ with $(c-\alpha)\mathcal{U}(\sl)/(c-\alpha)I$ mapping onto $(c-\alpha)A_\alpha$. Hence we also have the splitting of the sequence \begin{display}
0 \ar{r} & (c-\alpha)A_\alpha \ar{r} & A_\alpha \ar{r} & A_\alpha/\langle c-\alpha \rangle \ar{r} & 0
\end{display} and the statement follows.
\end{proof}

With these results, we have now determined the structure of $A_\alpha$ as a direct summand in the adjoint representation of $\mathcal{U}_0(\sl)$.

\begin{theorem-N}
The structure of a block $A_0\simeq\bigoplus_{i=0}^{(p-1)/2} P_{0,2i}$ as an $\sl$-module. For $\alpha \neq 0$ the adjoint representation $A_\alpha$ is the direct sum of the two components in Proposition \ref{projective_decomp}.
\end{theorem-N}

\begin{proof}
Follows from the previous results.
\end{proof}

\begin{remark}
The decomposition matches with the decomposition of the adjoint representation for the small quantum group $\sl$ at the $p$th root of unity established in \cite{ostrik}.
\end{remark}

\section{Acknowledgments}
The authors thank Roman Bezrukavnikov for introducing the problem and many helpful suggestions. The authors are grateful to the MIT SPUR program and its organizers and, especially, its advisors Ankur Moitra and David Jerison for helpful conversations. The first author was funded by RFBR, project number 19-31-90078.
\medskip

\bibliographystyle{acm}
\bibliography{citations.bib}

\end{document}